\documentclass[11pt]{article}
\usepackage{geometry}                
\usepackage{epstopdf}

\usepackage{amsfonts,amsmath,latexsym,amssymb,euscript,graphicx,units,mathrsfs,
  amsthm}
\usepackage{bbm, graphicx}
\usepackage{caption}

\RequirePackage[colorlinks,citecolor=blue,urlcolor=blue]{hyperref}

\newtheorem{theorem}{Theorem}[section]
\newtheorem{lemma}[theorem]{Lemma}
\newtheorem{proposition}{Proposition}[section]
\newtheorem{cor}{Corollary}[section]

\newtheorem{remark}[theorem]{Remark}
\newtheorem{definition}[theorem]{Definition}
\newtheorem{example}[theorem]{Example}

\newcommand{\dw}{\mathop{d_{\mathrm{W}}}}
\newcommand{\IE}{\mathbbm{E}}

\def\m{\medskip}

\def\bi{\begin{itemize}}
\def\ei{\end{itemize}}

\usepackage{pdfpages}

\title{Stein's method and approximating the multidimensional quantum harmonic oscillator}
\author{Ian W. McKeague and Yvik Swan}
\date{}
\begin{document}
\captionsetup[figure]{labelfont={bf},labelformat={default}, labelsep=space, name={Fig.}}

\maketitle

\begin{abstract}
  Stein's method is used to study discrete representations of multidimensional distributions that arise as approximations of states of quantum harmonic oscillators.  These representations model how quantum effects result from the interaction of finitely many classical ``worlds," with the role of sample size played by the number of worlds.  Each approximation arises as the ground state of a Hamiltonian involving a particular interworld potential function.  Such approximations have previously been studied for one-dimensional quantum harmonic oscillators, but the multidimensional case has remained unresolved.  Our approach, framed in terms of spherical coordinates, provides the rate of convergence of the discrete approximation to the ground state in terms of Wasserstein distance.  The fastest rate of convergence to the ground state is found to occur in three dimensions.  This result is obtained using a discrete density approach to Stein’s method applied to the radial component of the ground state solution.
      
  \end{abstract}
 {\small \noindent Mathematics Subject Classification (MSC2020):  60B10, 81Q05, 81Q65\\
  Keywords: Coupling, Hamiltonian, point configurations, Stein's method, Wasserstein distance\\

\noindent  {\bf Author affiliations:}\\
Ian W. McKeague: Columbia University, New York, U.S.A.  \\
Yvik Swan:  Universit\'e libre de Bruxelles, Bruxelles, Belgium.\\

\noindent {\bf Corresponding author}: Yvik Swan\\

\noindent {\bf Acknowledgments}.
The authors thank Louis Chen, Davy Paindaveine, Erol Pek{\"o}z, Gesine Reinert and  Thomas Verdebout for helpful discussions.  \\

\noindent {\bf Declarations:}\\
Funding: The research of Ian McKeague was partially supported by National Science Foundation Grant DMS 2112938 and  National Institutes of Health Grant AG062401.  The research of Yvik Swan was supported by by grant CDR/OL J.0197.20 from FRS-FNRS. \\
Availability of data and material: provided in the Supplementary material.\\
Code availability: provided in the Supplementary material.}

\newpage 

\tableofcontents

 \section{Introduction}

 The many interacting worlds (MIW) approach  to quantum mechanics due to Hall et al.\   \cite{Hal2014} posits a Hamiltonian for a one-dimensional harmonic oscillator of the form
 $$H_1({\bf x},{\bf p}) =  E({\bf p}) +V_1({\bf x}) + U_1({\bf x}), $$
where the locations of  particles (having unit mass) in  $N$ worlds are specified by  ${\bf x} =(x_1,\ldots , x_N)$  with $x_1>x_2>\ldots > x_N$, and their momenta by   ${\bf p} =(p_1,\ldots,p_N)$.
Here  $E({\bf p}) = \sum_{n=1}^N p_n^2/2$ is the  kinetic energy,
$V_1({\bf x}) = \sum_{n=1}^N x_n^2$
is the  potential energy (for the parabolic trap), and
$$U_1({\bf x}) = \sum_{n=1}^N\left({1\over x_{n+1}-x_n} -{1\over x_{n}-x_{n-1}}\right)^2$$
is  called the ``interworld" potential, where
$x_0=\infty$ and $x_{N+1}=-\infty$.

The ground state  particle locations $x_n$ are symmetric (about the origin) and satisfy \cite{Hal2014}  \begin{equation}\label{e1}x_{n+1}=x_n-\frac{1}{x_1+\cdots +x_n}.\end{equation}
McKeague and Levin \cite{mckeague2016}  showed that the  empirical distribution of the resulting ground state solution  $\{x_n, n=1,\ldots, N\}$  tends to standard Gaussian, and conjectured that the optimal rate of convergence in Wasserstein distance is of order $\sqrt{\log N}/N$; this conjecture was recently proved by 
Chen and Th{\`a}nh \cite{chen2020optimal}. 

For the first excited state, McKeague, Pek{\"o}z and Swan \cite{MPS19} showed that an extension of  the above interworld potential leads to the two-sided Maxwell distribution as the limit (agreeing with the classical quantum harmonic oscillator).
Ghadimi et al.\ \cite{e20080567} studied non-locality by introducing other extensions of the  MIW interworld potential for the first excited state using higher-order smoothing methods.    

In this article we study the question of how the MIW approach can be extended to general $d$-dimensional settings.  Herrmann et al.\ \cite{Her2018} examined the case $d=2$  and proposed using Delaunay triangulations and Voronoi tesselations of point configurations to extend the notion of the interworld potential. They developed a numerical algorithm to estimate the resulting ground state configuration, but the question of whether it is possible to establish an asymptotically valid approximation of such a solution to the classical quantum harmonic oscillator  ground state solution was not addressed.  

Our  contribution is two-fold: (1) we introduce new interworld potentials that apply in general $d$-dimensional settings and that lead to tractable ground state solutions (as well as some excited state solutions) for the corresponding MIW Hamiltonian; (2) we provide a new version the discrete density approach to Stein’s method that furnishes  upper bounds on the convergence rate of the  radial components of these ground state solutions. Evaluating these bounds numerically, leads to optimal rates of convergence (in terms of Wasserstein distance) of the full  MIW  configurations that apply to general multidimensional settings.  In contrast to \cite{mckeague2016,MPS19}, which rely on the coupling version of Stein's method, the present approach leads to optimal rates.  In the Supplementary material we  explore the types of upper bounds that can be obtained using the coupling approach in the multidimensional case.

The paper is organized as follows.  Section 2 develops the proposed interworld potentials, first for the two- and three-dimensional cases, then for the general $d$-dimensional setting, including some that apply to excited states.   The resulting MIW ground state solutions are described in terms of their spherical coordinates, and plots of the solutions are provided for $d=2$ and $3$.  An upper bound on the Wasserstein distance between two probability measures that have independent radial and directional components is provided at the end of Section 2.   Section 3 restricts attention to the radial component and develops the discrete density approach to Stein’s method mentioned above.   In Section 4 we wrap up by providing  optimal rates for  full $d$-dimensional ground state solutions.  
The Supplementary material includes discussion of the coupling approach mentioned above along with computer code.

\section{Extending interworld potentials to higher dimensions}
\label{sec:interworld}

In this section we introduce interworld potentials that lead to satisfactory MIW approximations to the the ground states and excited states  of the  $d$-dimensional isotropic quantum harmonic oscillator.   First we consider the two-dimensional case.

\subsection{Two-dimensional case}
\label{2D}

When $d=2$, the idea  is to couch the problem in terms of polar coordinates, which allows the interworld potential to adapt to the basic geometry of the problem, specifically via a separation of its angular (directional) and radial components.  
To that end,  a configuration ${\bf x} $ of points in $\mathbb{R}^{2}$ is specified in terms of (signed) polar coordinates as $$\{(r_{nj},\theta_j)\colon \    n=1,\ldots,N_j, \ j=1,\ldots,M\},$$ where 
 there are $M=M({\bf x} )\ge 1$ (distinct) angular coordinate values satisfying  $0\le \theta_1< \theta_2 < \ldots < \theta_M< \pi$.  The number of points $N_j$ in any direction  $\theta_j$ is assumed to satisfy $N_j\ge 2$, to avoid the possibility of a ``degenerate" radial contribution from a single point at the origin.  The radial coordinates $r_{nj}$ in  direction  $\theta_j$  
in this representation are signed (can be negative as well as positive) and  satisfy $r_{1j}> r_{2j} >\ldots > r_{N_j j}$; a point $(r,\theta)$ with  $r<0$  is understood to correspond to the point with polar coordinates $(|r|,\theta+\pi)$. The total number of points  is $N=N_1+\ldots +N_M$, with the understanding that  points  at the origin arising from different directions  are    distinct elements  of the configuration.   
 
 We propose the following ansatz for the interworld potential:
 \begin{align}
\label{potential}  U_2({\bf x}) & = 4 \sum_{j=1}^M \sum_{n=1}^{N_j}\left[ {1\over R_2(r_{n+1,j})- R_2(r_{nj})} -{1\over R_2(r_{nj})- R_2(r_{n-1,j})}\right]^2 r_{nj}^2 \notag \\
&\ \ \ +  \pi \sum_{j\sim k}  {1 \over |\theta_j-\theta_k|_\pi} + {N^2 \over M^2} L\left(\sum_{j\sim k} |N_j-N_k|\right), 
\end{align}
where the notation $R_d(r) =|r|^d\, {\rm sign} (r)$ for general integers $d\ge 2$, and $L(x)= \max(1,x/2)$, will be used frequently in the sequel.  By convention we set $R_d(r_{0j})=\infty$ and  $R_d(r_{N_j+1,j})=-\infty$.  The summation in the second and third terms is over all pairs of angular values that are neighbors (mod $\pi$), and $|\cdot |_\pi$ denotes absolute value mod $\pi$;  in the sequel we also use such notation when $\pi$ is replaced by other positive real numbers, depending on the context.

The intuition behind  the first term in this ansatz comes from the well-known representation of the  standard Gaussian density in two-dimensions in terms of  polar coordinates.  That is, the angular coordinate (in the sense we defined above) can be taken as uniformly distributed on $[0,\pi]$, independent of the (signed) radial coordinate, which has the (two-sided) Rayleigh density
$p(r) = b(r) \varphi(r)$, where $\varphi$ is a standard normal density and $b(r)=\sqrt{\pi/2}|r|$, $r \in \mathbb{R}$.

The first term in $U_2({\bf x})$ is based on a derivation given in   \cite{MPS19}, in which an interworld potential is proposed for densities of the general form $p(x) = b(x) \varphi(x)$, $x \in \mathbb{R}$. 
The  focus in that paper was on the case of the two-sided Maxwell distribution for which  $b(x)=x^2$. The  general ansatz for the interworld potential of a one-dimensional $N$-point configuration $x_1>x_2>\ldots > x_N$ was proposed to be
  \begin{equation}
\label{potentialb}  U_b({\bf x}) = \sum_{n=1}^N\left[ {1\over B(x_{n+1})- B(x_n)} -{1\over B(x_{n})- B(x_{n-1})}\right]^2 b(x_n)^2,
\end{equation}
where  $B(x)=\int_0^x b(t)\, dt$, $B(x_0)=\infty$ and $B(x_{N+1})=-\infty$.
When $b(x)$ is proportional to $|x|^k$ for some positive integer $k$, as for the Rayleigh density ($k=1$), it can be shown that the minimizer of the Hamiltonian $U_b({\bf x}) +V_1({\bf x})$ is a symmetric solution of the recursion
\begin{equation}
\label{genrec}
 B(x_{n+1})=  B(x_{n})- \left(\sum_{i=1}^n \frac{x_i}{b(x_i)}\right)^{-1}.
\end{equation}
For the Rayleigh distribution,  $b(x)=\sqrt{\pi/2}|x|$, so the interworld potential (\ref{potentialb}) becomes
$$U_b({\bf x})= 4 \sum_{n=1}^N\left[ {1\over R_2(x_{n+1})- R_2(x_n)} -{1\over R_2(x_{n})- R_2(x_{n-1})}\right]^2 x_n^2.$$
From (\ref{genrec}), the ground state is then a symmetric solution to the recursion equation 
\begin{equation}
\label{Ralrec}
R_2(x_{n+1})=R_2(x_{n})- 2\left(\sum_{i=1}^n {\rm sign}(x_i)\right)^{-1}.
\end{equation}
Further, it follows from analogous arguments in  \cite[Section 2]{MPS19} that the ground state value of the Hamiltonian in this case  is $4(N-1)$.  Fig.\  \ref{hist} compares the 
empirical distribution of the solution of the above recursion (for $N=22$ points) with the two-sided Rayleigh density, showing that the agreement is remarkably accurate.

\begin{figure}[!ht]
\begin{center}
      \includegraphics[scale=.4]{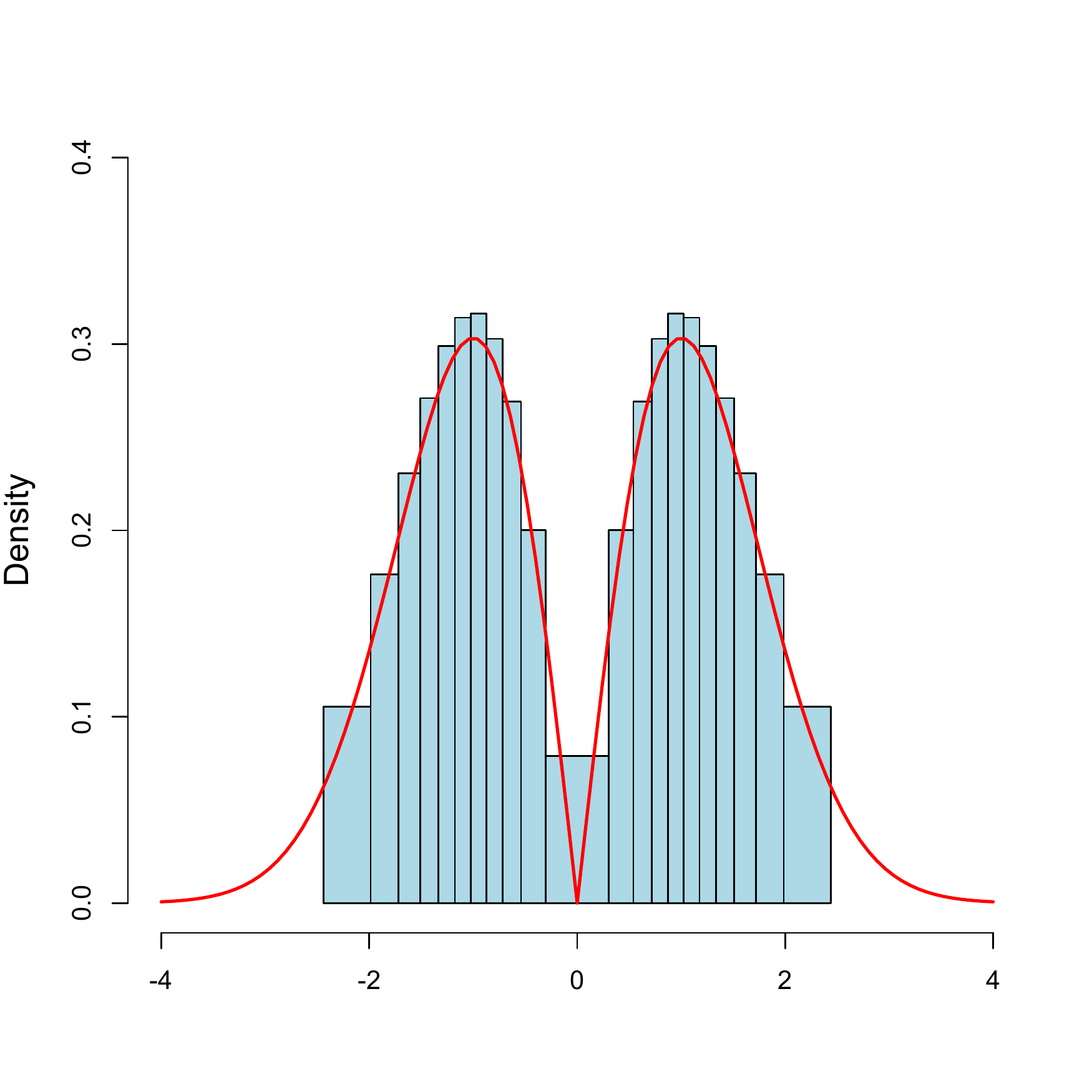}
 \caption{Histogram of the symmetric solution of the recursion (\ref{Ralrec}) for $N=22$  compared with the two-sided Rayleigh density; the breaks in the histogram are $x_1, \ldots,x_N$}
   \label{hist}
 \end{center}
   \end{figure}

Returning now to the  two-dimensional setting with  Hamiltonian
\begin{equation}
\label{2dHam}
 H_2({\bf x}) = U_2({\bf x}) + \sum_{j=1}^M \sum_{n=1}^{N_j}r_{nj}^2,\end{equation}
 note that the $M$ components of the first term of the interworld potential $U_2({\bf x})$ can be combined with the corresponding potential energy terms and separately minimized (using the  recursion (\ref{Ralrec})), giving a combined contribution of 
$$ \sum_{j=1}^M  4(N_j-1) =4(N-M)$$ to the Hamiltonian.  

It remains to minimize the sum of the second and third terms in $U_2({\bf x})$  to furnish the complete ground state solution.   First consider $M$ as fixed and note that by Cauchy--Schwarz
\begin{align}
  M & =\sum_{j\sim k} { |\theta_j-\theta_k|_\pi^{1/2}\over |\theta_j-\theta_k|_\pi^{1/2}}  \notag  \le \left(\sum_{j\sim k} { |\theta_j-\theta_k|_\pi}\right)^{1/2}  \left(\sum_{j\sim k} {1\over |\theta_j-\theta_k|_\pi}\right)^{1/2}  \notag \\
  & 
 =  \left(\pi \sum_{j\sim k} {1\over |\theta_j-\theta_k|_\pi}\right)^{1/2} \notag,
\end{align}
so the minimum of the second term is  $M^2$, attained by setting $\theta_j=(j-1)\pi/M,\ j=1,\ldots, M$, because that produces  equality throughout the above display.  With $M\le N/2$ fixed, the third term in $U_2({\bf x})$ is  minimized  by distributing the $N$ points as evenly as possible among the $M$ directions (with at least two in each direction) and then allotting the remaining ($N$ mod $M$) points to a sequence of neighboring directions.  This  results in $N_j=N_k$ for all neighbors $j$ and $k$ except possibly for two  pairs of neighbors in which $|N_j-N_k|=1$.  

The  sum of the last two terms of $U_2({\bf x})$ therefore has the minimal value $M^2+ {N^2/ M^2}$, 
and when combined with the contribution $ 4(N-M)$ arising from  the first term in $U_2({\bf x})$ along with the potential energy, as discussed earlier, the ground state minimizes $4(N-M) + M^2 +N^2/M^2$ as a function of $M$. Asymptotically, this minimum is attained  when $M\sim \sqrt N$. It follows also that the empirical distribution of  $\{N_j/N, j=1,\ldots, M\}$ in the ground state converges weakly to uniform on $(0,1)$ as $N\to \infty$.

\begin{figure}[!ht]
\begin{center}
 \includegraphics[scale=.3]{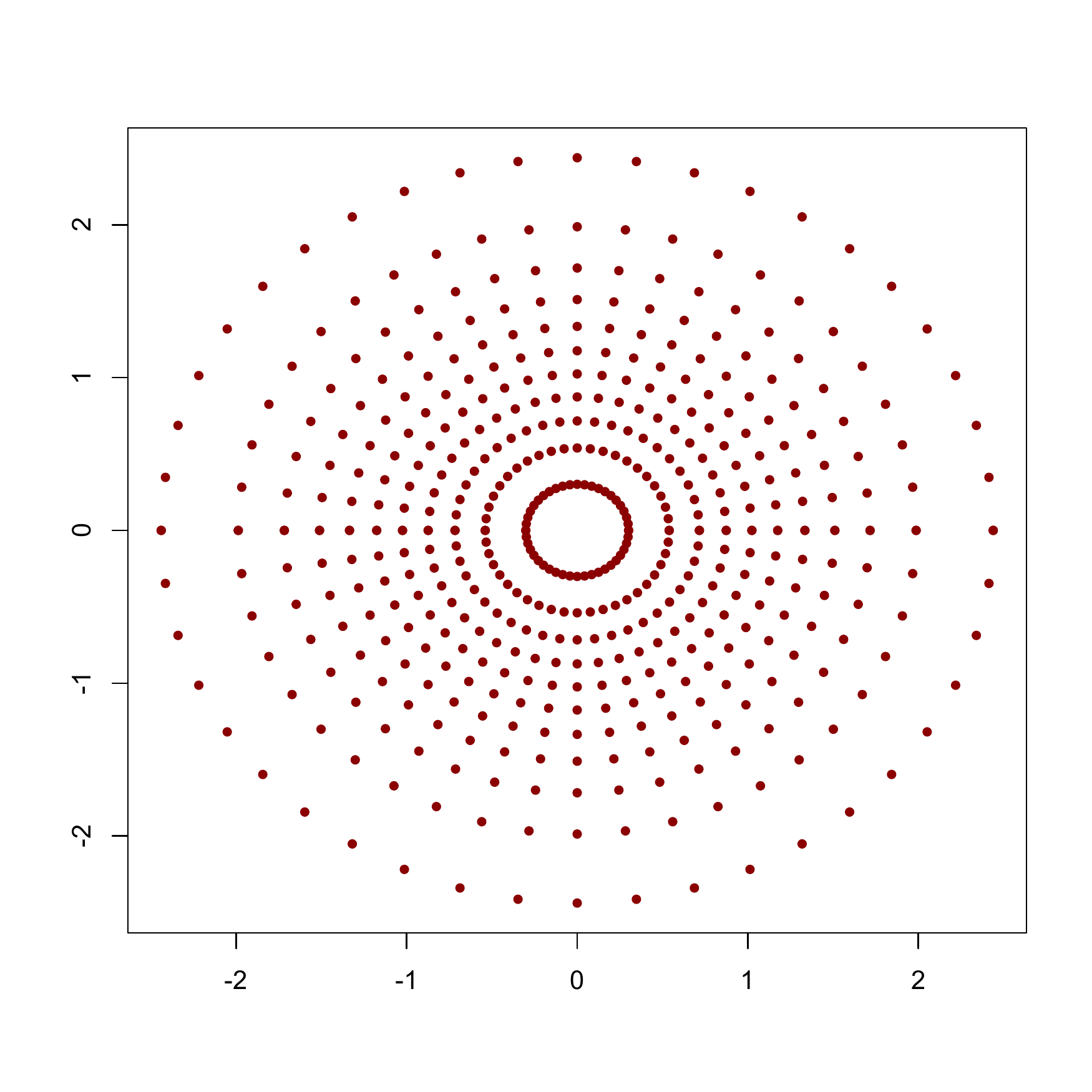} \includegraphics[scale=.33]{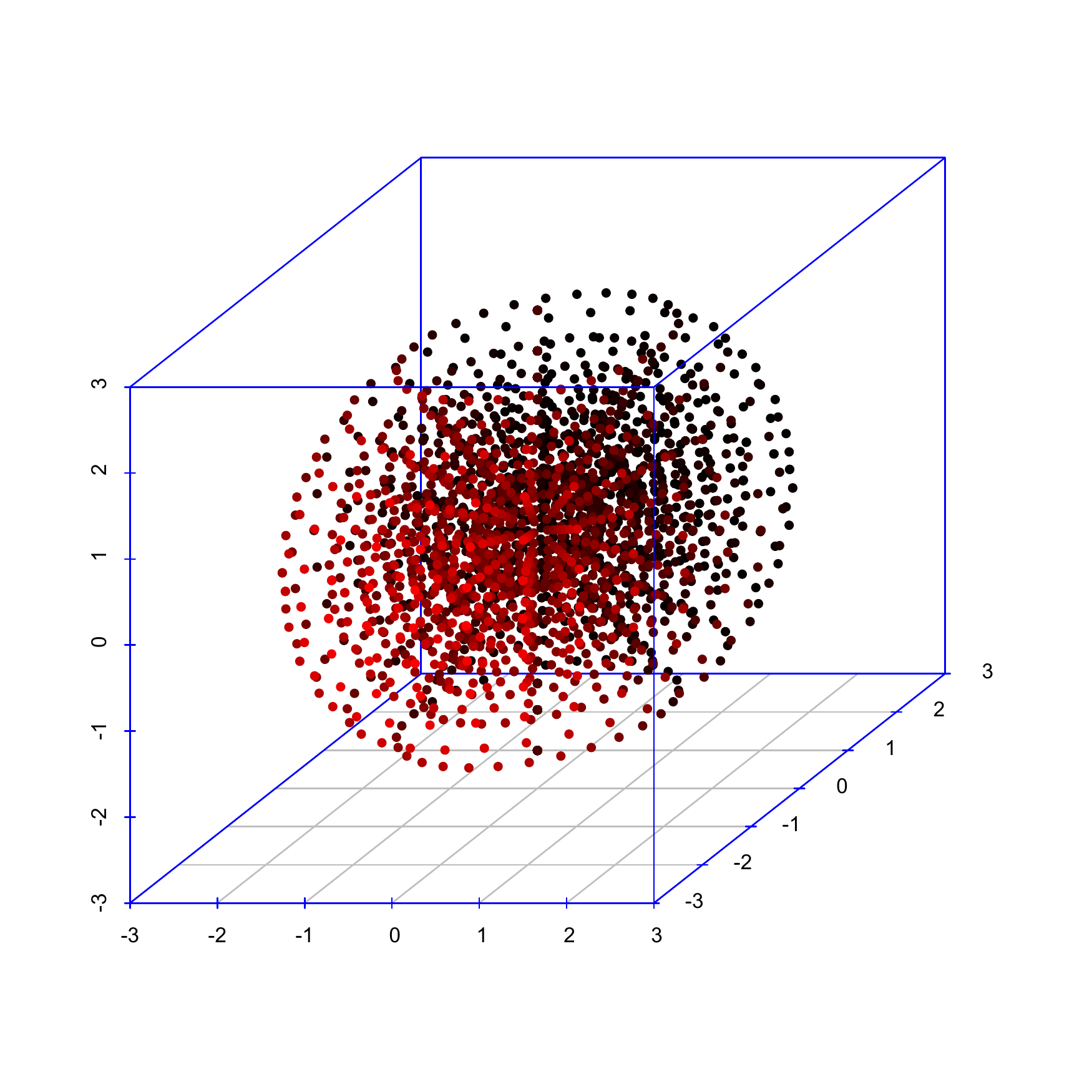}
 \caption{Ground states:  $d=2$ (left),  $N=22^2=484$ points, $M=22$ directions and $N_j=22$ points in each direction;   $d=3$ (right), $N=2744$ points, $M=7$ shells (polar angles in $[0,\pi/2]$), $K_j=28$ azimuthal angles in each shell, and $N_{jk} =14$ points in each radial direction 
 }
   \label{scatter}
 \end{center}
   \end{figure}

Combined with the fact that the empirical distribution (conditional on $M$) of any set of minimizing angular coordinates  $\{\theta_j, j=1,\ldots, M\}$ converges weakly to uniform on $(0,\pi)$ as $M\to \infty$, we conclude that the  empirical distribution of the angular coordinates of the $N$ points in the ground state has the same limit as $N\to \infty$.  Thus, provided we can show that the  empirical distribution of the radial  coordinates of the ground state converges to the Rayleigh distribution, it will follow that the full empirical distribution of the ground state (as illustrated by the left panel of Fig.\  \ref{scatter}) converges to the standard 2D Gaussian distribution.

\subsection{Interworld potentials  for  $d\ge 3$}
\label{3D}

In  the three-dimensional case, a  configuration  ${\bf x}$ of points in $\mathbb{R}^{3}$ is  specified in terms of (signed) spherical coordinates as  \begin{align}\label{spco}
\{(r_{njk},\theta_j,\phi_{jk})\colon \    n=1,\ldots,N_{jk}, \ j=1,\ldots,M,\  k=1,\ldots,K_j\},\end{align} where there are (distinct) directions corresponding to points with polar angles $0\le \theta_1< \theta_2 < \ldots < \theta_M\le \pi/2$ and azimuthal angles  $0\le \phi_{j1}< \phi_{j2} < \ldots < \phi_{jK_j} < 2\pi$, along with $N_{jk}\ge 2$ points in each direction  $(\theta_j,\phi_{jk})$
 determined by their  (signed)  radial distances  $r_{1jk}< \ldots < r_{N_{jk}jk}$.   There are a total of $N_{j\boldsymbol{\cdot}} =\sum_{k=1}^{K_j} N_{jk}$ points corresponding to the $K_j$ directions  in the $j$th ``shell" with  polar angle $\theta_j$.   Over all  $M$  shells, there are a total of $N=\sum_{j=1}^{M}  N_{j\boldsymbol{\cdot}}$ points arising from $K=\sum_{j=1}^{M} K_j$ different directions.   
 
 The spherical coordinates 
 of a 3-dimensional standard Gaussian random vector are independent.  The signed radial component has the two-sided Maxwell distribution with density $ r^2\varphi(r),r\in \mathbb{R}$, and  the azimuthal angle  is uniformly distributed on $[0,2\pi)$.  The polar angle  has 
 cdf $1-\cos(\theta)$, $\theta \in [0,\pi/2]$. 
Motivated by similar considerations as in the two-dimensional case, the proposed interworld potential is taken to be
 \begin{align}
\label{pot3}  U_3({\bf x}) & = 9 \sum_{j=1}^M  \sum_{k=1}^{K_j}\sum_{n=1}^{N_{jk}}\left({1\over r_{n+1,jk}^3-r_{njk}^3} -{1\over r_{njk}^3-r_{n-1,jk}^3}\right)^2r_{njk}^4 \notag \\
&\ \ \ +  \sum_{i\sim j}  {1 \over |\cos(\theta_{i})-\cos(\theta_{j})|} +
  {\pi \over 2}\sum_{j=1}^M\sum_{k\sim l}  {1 \over |\phi_{jk}-\phi_{jl}|_{2\pi}}   \\
&\ \ \ +  \sum_{j=1}^M {N_{j\boldsymbol{\cdot}}^2\over K_j^2} \, L\!\!\left(\sum_{k\sim l} |N_{jk}-N_{jl}|\right)+ {N\over 4M} \, L\!\!\left(\sum_{i\sim j}  | N_{i\boldsymbol{\cdot}}-N_{j\boldsymbol{\cdot}}|\right).\notag 
\end{align}
The last term above has a distinctive form that is not present in $U_2({\bf x})$.  The denominator in this term is designed to control the number of polar angles ($M$) in the second term, which in turn needs to be balanced with the $K_j$, since the polar angles range over an interval only one fourth the length of that for the azimuthal angles.  Numerical experiments show that the resulting directional component in the ground state is close to uniformly distributed over the sphere, as illustrated in the right panel of Fig.\  \ref{scatter}.

We now seek to minimize the  Hamiltonian
\begin{align} \label{H3}H_3({\bf x}) = U_3({\bf x}) + \sum_{j=1}^M\sum_{k=1}^{K_j} \sum_{n=1}^{N_{jk}}r_{njk}^2.
 \end{align}
The components with distinct values of $(j,k)$ in the first term of the interworld potential $U_3({\bf x})$ can be combined with the corresponding potential energy terms and separately minimized (in terms of the  recursion (\ref{genrec}) with $b(x) = x^2$ that was studied in \cite{MPS19}), giving a combined contribution of 
$$  \sum_{j=1}^M\sum_{k=1}^{K_j}6(N_{jk}-1)= 6(N-K)$$ to the Hamiltonian.

The next step is  to minimize each of the last three terms in (\ref{pot3}) for fixed  $M$, $K_1,\ldots, K_M$. Using a similar Cauchy--Schwarz  argument to what we used to minimize the second term in (\ref{potential}), the second term in   (\ref{pot3}) is 
minimized for \begin{align}\label{pcos} \theta_j = \cos^{-1} (1-(j-1)/(M-1)),\ j=1,\ldots, M, \end{align}
and the minimum is $(M-1)^2$. Similarly, the minimum of the third term in (\ref{pot3}) is attained by setting 
$$\phi_{jk}=2\pi(k-1)/K_j,\ k=1,\ldots, K_j,\ j=1,\ldots,M,$$ and the minimum is $\sum_{j=1}^M K_j^2/4$.
The minimum of the fourth term given fixed values of the $K_j$ and $N_{j\boldsymbol{\cdot}}$ is found using the same argument as in the two-dimensional case for the third term in (\ref{potential}), resulting in the minimum 
$ \sum_{j=1}^M {N_{j\boldsymbol{\cdot}}^2/ K_j^2}$
being attained when $N_{jk}\sim N_{j\boldsymbol{\cdot}}/K_j$ for $k=1,\ldots, K_j$.  This reduces the minimal value of the Hamiltonian to
$$6(N-K) + (M-1)^2 + \sum_{j=1}^M [K_j^2/4 + {N_{j\boldsymbol{\cdot}}^2/ K_j^2}]+ {N\over 4M} \, L\!\!\left(\sum_{i\sim j}  | N_{i\boldsymbol{\cdot}}-N_{j\boldsymbol{\cdot}}|\right).$$  For fixed $M$, this expression   is minimized 
  when
 $K_j\sim \sqrt{2 N_{j\boldsymbol{\cdot}}}$ as $N_{j\boldsymbol{\cdot}}\to \infty$, and it remains
 to minimize
 $$7N  +(M-1)^2 + {N\over 4 M} \, L\!\!\left(\sum_{i\sim j}  | N_{i\boldsymbol{\cdot}}-N_{j\boldsymbol{\cdot}}|\right).$$
Asymptotically, this expression is minimized by  $N_{j\boldsymbol{\cdot}}\sim N/M$  and setting the number of shells $M\sim N^{1/3}/2$. This implies $$N_{jk}\sim N_{j\boldsymbol{\cdot}}/K_j \sim \sqrt{N_{j\boldsymbol{\cdot}}/2}\sim \sqrt {N/(2M)}\sim N^{1/3},\ \ K_j\sim 2N^{1/3}.$$  We conclude that the ground state asymptotically
consists of  $K_j=2N^{1/3}$ directions in each of  $M=N^{1/3}/2$ shells, with $N_{jk}=N^{1/3}$ points falling in  each direction.  An example with $N=2744$ points is provided in the right panel of Fig.\  \ref{scatter}.

\bigskip
The extension to  general  $d\ge 3$ is straightforward.
The  radial component of the $d$-dimensional standard Gaussian distribution has  density 
$p(r) =c_d |r|^{d-1}\varphi(r),\  r\in \mathbb{R}$, where $c_d$ is a normalizing constant, and is independent of the (signed) directional component, which is uniformly distributed (in the sense of Hausdorff measure) over the unit-hemisphere $ \{ x \in \mathbb{R}^d\colon { x}_1\ge 0,  |{x}| =1\}$. The configuration $\bf x$ is now specified using signed {\it hyperspherical} coordinates, precisely as in  (\ref{spco}), except  each polar angle $\theta_j$ is now required to be a {\it vector} $\theta_j=(\theta_{jl})$  of $d-2$ polar angles, with each component satisfying  $0\le \theta_{1l}< \theta_{2l} < \ldots < \theta_{Ml}\le \pi/2$, $l=1,\ldots,d-2$.  The notation for the azimuthal angles, the $M$ shells and the $K$ directions remain the same.

The $d-2$ polar angles are iid under the $d$-dimensional standard Gaussian, and their cdf can be expressed using a simple formula  for the area of a hyperspherical cap \cite{li2007}:
$$F_d(\theta) = B\!\left(\sin^2\theta, {d-1\over 2}, {1\over 2}\right), \ \theta\in [0,\pi/2],$$ 
where $B(x; a, b) =\int_0^x t^{a-1} (1-t)^{b-1} \, dt $ is the incomplete Beta function.
In particular, $F_4(\theta)= (2\theta-\sin(2\theta))/\pi$ and $F_5(\theta) =1-(3/2)\cos \theta + (1/2) \cos^3 \theta$. This leads to the interworld potential
 \begin{align}
\label{potd} &  U_d({\bf x})= \nonumber \\ & d^2 \sum_{j=1}^M  \sum_{k=1}^{K_j}\sum_{n=1}^{N_{jk}}\left({1\over R_d(r_{n+1,jk})-R_d(r_{njk})} -{1\over R_d(r_{njk})-R_d(r_{n-1,jk})}\right)^2r_{njk}^{2(d-1)} \notag \\
&\ \ \ +  {1\over (d-2)}\sum_{l=1}^{d-2}\sum_{i\sim j}  {1 \over |F_d(\theta_{il})-F_d(\theta_{jl})|} +
  {\pi\over 2} \sum_{j=1}^M\sum_{k\sim l}  {1 \over |\phi_{jk}-\phi_{jl}|_{2\pi}}   \\
&\ \ \ +  \sum_{j=1}^M {N_{j\boldsymbol{\cdot}}^2\over K_j^2} \, L\!\!\left(\sum_{k\sim l} |N_{jk}-N_{jl}|\right)+ {N\over 4(d-2)M^{d-2}} \, L\!\!\left(\sum_{i\sim j}  | N_{i\boldsymbol{\cdot}}-N_{j\boldsymbol{\cdot}}|\right).\notag 
\end{align}
 It follows by a routine extension of the argument we used in the  $d=3$ case that  the ground state 
asymptotically
consists of  $K_j=2N^{1/2 -1/(2d)}$ directions in each of  $M=N^{1/d}/2$ shells, with $N_{jk}=N^{1/2-1/(2d)}$ points falling in  each direction.  

A possible alternative approach is to specify the directions in the unit-hemisphere by a minimal Riesz energy point configuration, without reference to polar and azimuthal components.  As surveyed in \cite{BRAUCHART2015}, such configurations can provide asymptotically uniformly distributed point sets with respect to surface area (Hausdorff measure), with important applications in quasi-Monte Carlo, approximation theory and material physics.  However, there does not seem to be a way of linking the numbers of radial points with the directions obtained from minimizing Riesz energy that would lead to a Gaussian approximation.  In contrast, the Hamiltonian based on the interworld potential (\ref{potd}) is readily minimized by following the same argument we used to minimize (\ref{pot3}); the only difference in the solution is that $\cos^{-1}$ in (\ref{pcos}) is replaced by $F_d^{-1}$ in the expression of each polar angle $\theta_{jl}$.  Moreover, as explained in the following sections, our proposed approach leads to explicit rates of convergence (in terms of Wasserstein distance) of the empirical measure of the ground state to standard Gaussian.
  
 \subsection{Interworld potentials for excited states}\label{excited}
 
The eigenstates of the  classical $d$-dimensional isotropic quantum harmonic oscillator  consist of  products of $d$ one-dimensional  eigenfunctions, with separate euclidean coordinates in each component.  In this section we show how the approach in the previous sections extends naturally to the case when some of these one-dimensional components are in excited (higher energy) states.  The various eigenstates of the full system are described by vectors of  quantum numbers indicating the energy level of each component (with 0 indicating the ground state).  When expressed in spherical or polar coordinates, there is a separation of variables in the various eigenfunctions, which implies that the corresponding densities have independent contributions from the radial and angular components.   This allows a separation of the radial and angular components in the interworld potential, as we now make explicit in examples of excited states for $d=2$ and $3$.  The  idea readily extends to any excited state of the MIW quantum harmonic oscillator.

For $d=2$, the lowest three excited states have quantum numbers $(1,0)$, $(0,1)$ and $(1,1)$.  The density of the $(1,0)$-state in signed polar coordinates is proportional to $\cos^2(\theta) |r|^3 \varphi(r)$, $\theta\in [0,\pi)$, $r \in \mathbb{R}$.    The cdf of the angular component is $G_{2}(\theta)=(\theta+\sin(\theta)\cos(\theta))/\pi$, and the density of the signed radial component is proportional to $|r|^3 \varphi(r)$.  This leads to the following interworld potential similar to (\ref{potential}):
\begin{align}
\label{potential2ex}  U_{2}({\bf x}) & = 16 \sum_{j=1}^M \sum_{n=1}^{N_j}\left[ {1\over R_4(r_{n+1,j})- R_4(r_{nj})} -{1\over R_4(r_{nj})- R_4(r_{n-1,j})}\right]^2 r_{nj}^6 \notag \\
&\ \ \ +  \sum_{j\sim k}  {1 \over |G_{2}(\theta_j)-G_{2}(\theta_k)|_1} + {N^2 \over M^2} L\left(\sum_{j\sim k} |N_j-N_k|\right). 
\end{align}
Following similar arguments to those used earlier, the Hamiltonian (\ref{2dHam}) based on (\ref{potential2ex})
is minimized  with the radial coordinates given by the symmetric solution to the recursion  (\ref{genrec}) with $b(x)=|x|^3$,  and setting $\theta_j=G_{2}^{-1}((j-1)/M)$ with $M\sim \sqrt N$.  Fig.\  \ref{2Dexcited} shows the  $(1,0)$ and $(1,1)$ excited states resulting from $N=484$  points. The $(0,1)$ state is the same as $(1,0)$ except rotated by 90 degrees.

\begin{figure}[!ht]
\begin{center}
 \includegraphics[scale=.33]{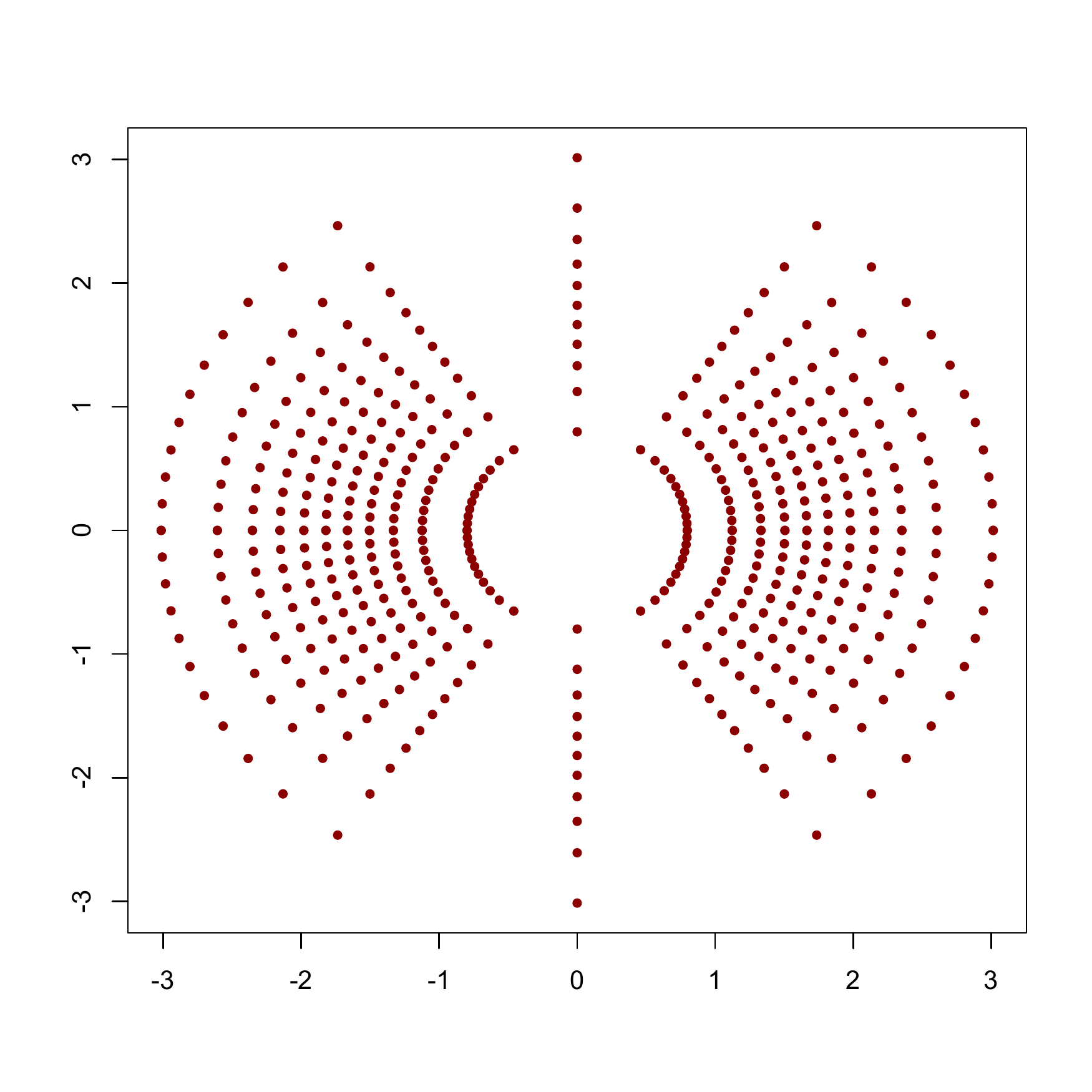}\includegraphics[scale=.33]{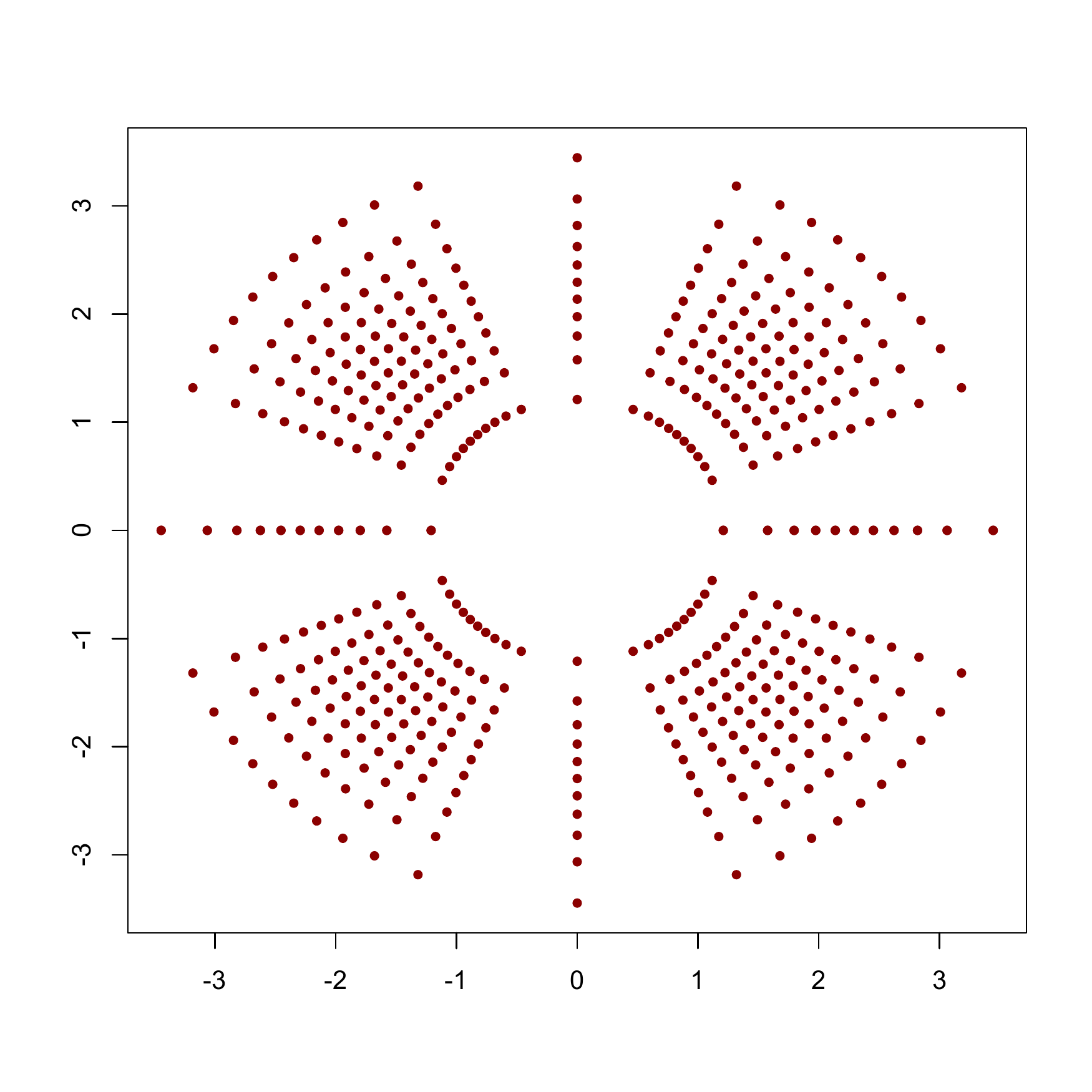}  
 \caption{Excited states $(1,0)$ (left), $(1,1)$  (right), for $d=2$, $N=22^2=484$ points,  $M=22$,  $N_j=22$, cf.\ the ground state in Fig.\  \ref{scatter}}
   \label{2Dexcited}
 \end{center}
   \end{figure}

For $d=3$, the density of the $(1,0,0)$ excited state in signed spherical coordinates is proportional to $\sin^3(\theta)\cos^2(\phi) r^4 \varphi(r)$, $\theta\in [0,\pi/2)$, $\phi\in [0,2\pi)$, $r \in \mathbb{R}$.
The cdf of the polar angle is $G_{3}(\theta) =(\cos(3\theta)-9\cos(\theta))/8$, and the cdf of the azimuthal angle is $A(\phi)=(\phi-\sin(\phi)\cos(\phi))/(2\pi)$.  The interworld potential is similar to (\ref{pot3}):
\begin{align*}
 U_3({\bf x}) & = 25 \sum_{j=1}^M  \sum_{k=1}^{K_j}\sum_{n=1}^{N_{jk}}\left({1\over r_{n+1,jk}^5-r_{njk}^5} -{1\over r_{njk}^5-r_{n-1,jk}^5}\right)^2r_{njk}^8 \notag \\
&\ \ \ +  \sum_{i\sim j}  {1 \over |G_{3}(\theta_{i})-G_{3}(\theta_{j})|} +
  {\pi \over 2}\sum_{j=1}^M\sum_{k\sim l}  {1 \over |A(\phi_{jk})-A(\phi_{jl})|_{2\pi}}   \\
&\ \ \ +  \sum_{j=1}^M {N_{j\boldsymbol{\cdot}}^2\over K_j^2} \, L\!\!\left(\sum_{k\sim l} |N_{jk}-N_{jl}|\right)+ {N\over 4M} \, L\!\!\left(\sum_{i\sim j}  | N_{i\boldsymbol{\cdot}}-N_{j\boldsymbol{\cdot}}|\right).\notag 
\end{align*}
 The configurations of the $(1,0,0)$, $(1,1,0)$ and $(1,1,1)$ excited  states resulting from  $N=2744$ points are displayed in Fig.\  \ref{3Dexcited110}.

\begin{figure}[!ht]
\begin{center}
 \includegraphics[scale=.22]{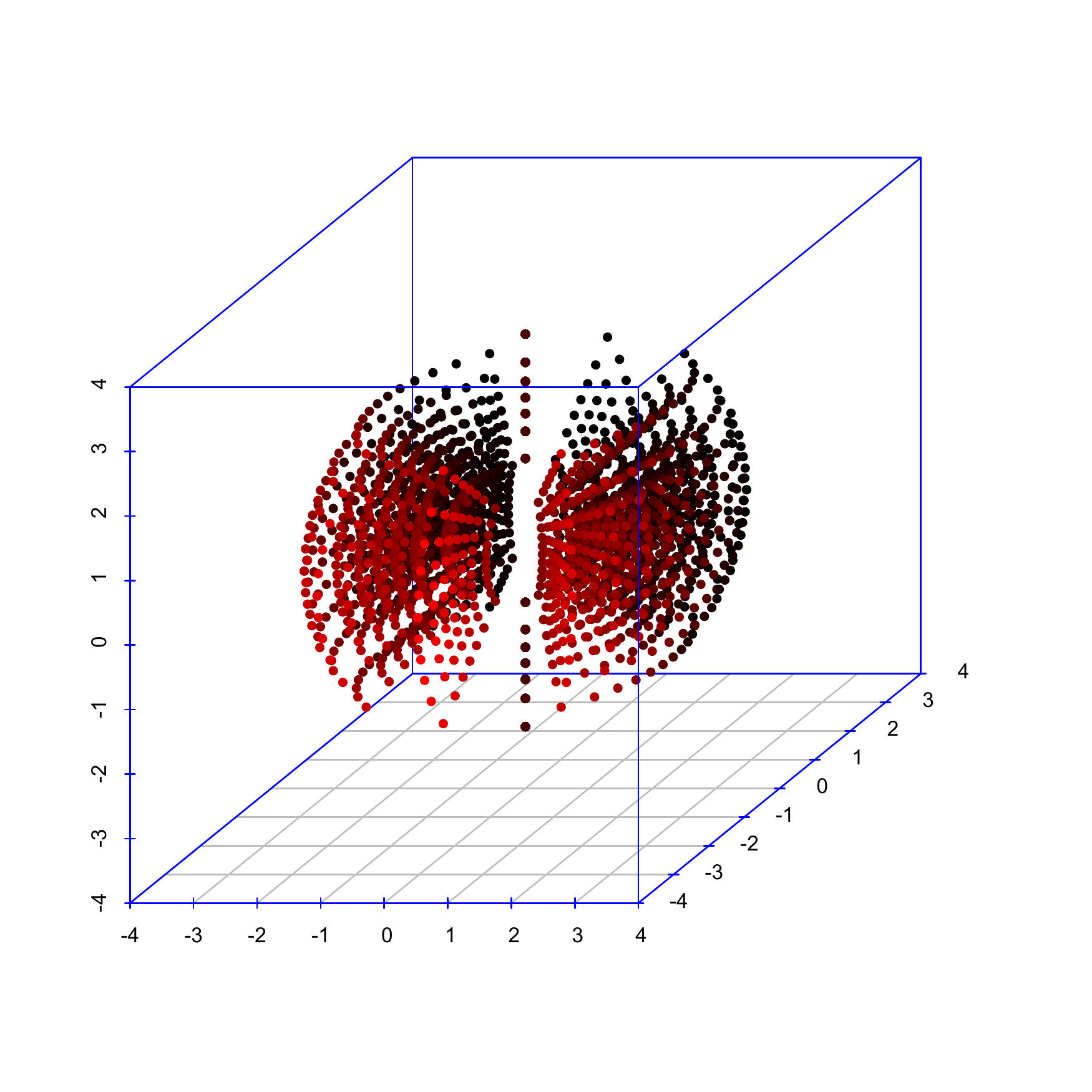}\includegraphics[scale=.22]{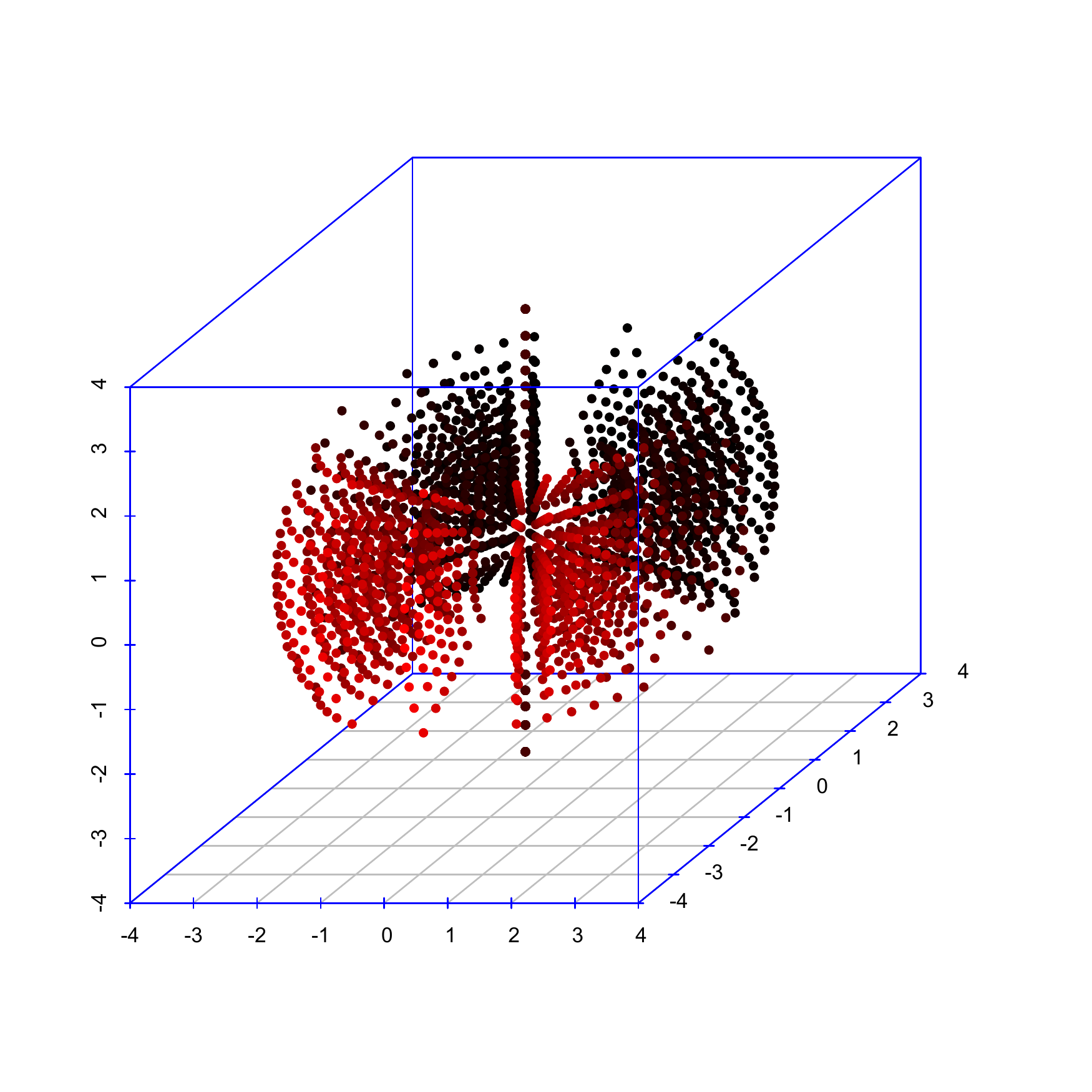} \includegraphics[scale=.22]{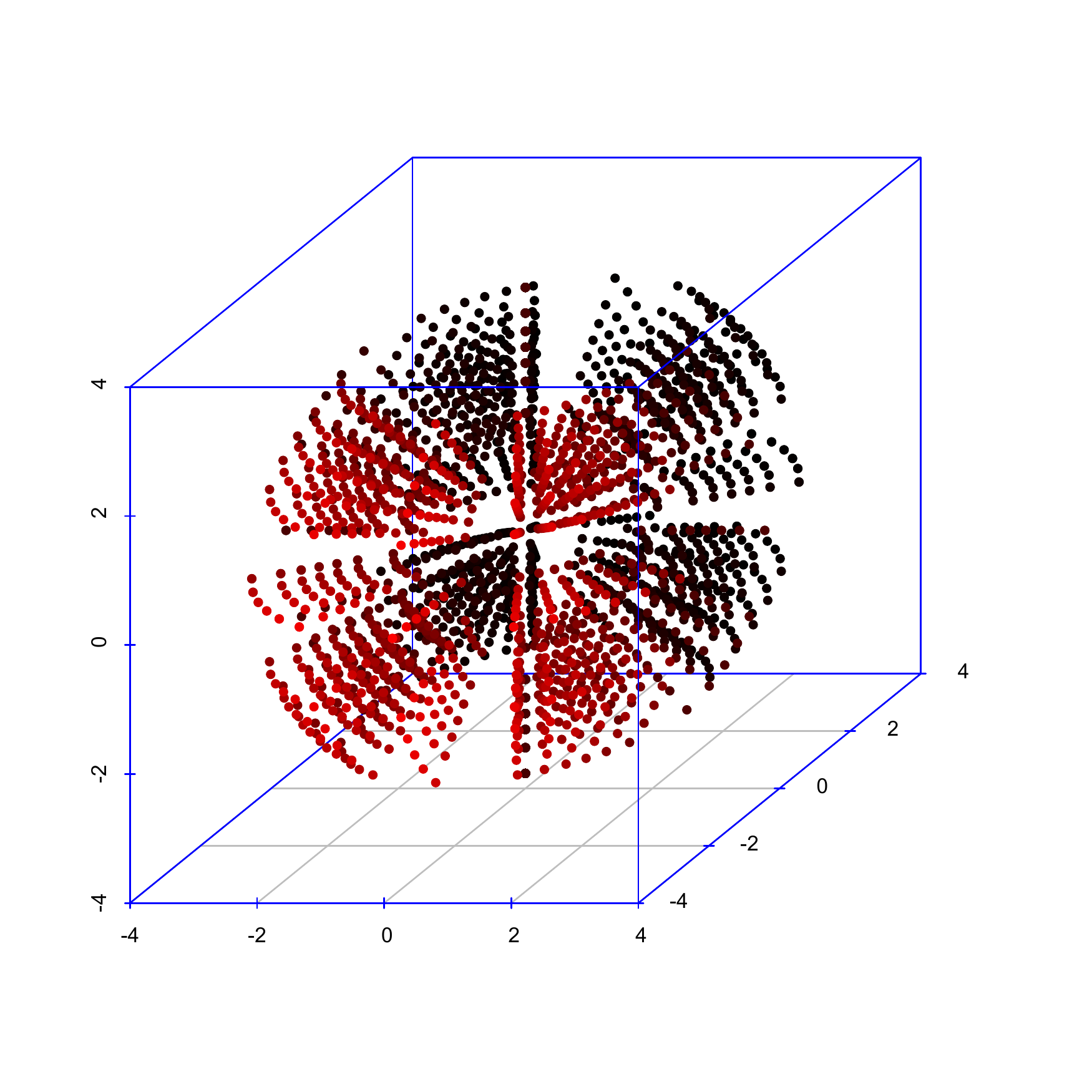}
 \caption{Excited states $(1,0,0)$ (left), $(1,1,0)$ (middle) and $(1,1,1)$ (right) for $d=3$, $N=2744$ points, $M=7$, $K_j=28$,  $N_{jk} =14$, cf.\ the ground state in the right panel of Fig.\  \ref{scatter} }
   \label{3Dexcited110}
 \end{center}
   \end{figure}

   \subsection{Wasserstein distance in spherical coordinates}
   \label{sec:wass-dist-spher}
The  Wasserstein distance between probability measures $\mu$ and $\nu$ on a Polish metric space $(S,\rho)$ can be defined equivalently as
\begin{equation*}
  \dw(\mu,\nu)=\sup_{h\in \mathcal{H}}\left|\int_ Sh\, d(\mu-\nu)\right| =\inf \mathbb{E} \rho(X,Y),
\end{equation*}
where $\mathcal{H}$ is the family of Lipschitz functions $h\colon S\rightarrow \mathbb{R}$  with $\mbox{Lip}(h)\leq 1$, and the (attained) infimum is over 
all coupled $S$-valued random elements $X\sim \mu$ and $Y\sim \nu$.  When $S=\mathbb{R}$ and  $\mu$ and $\nu$ have cdfs $F$ and $G$, their Wasserstein distance coincides with the $L_1$-distance  
$ \int |F(x)-G(x)| \, \mathrm{d}x$.   The following inequality bounds the Wasserstein distance  between  two product measures $\mu=\prod_{i=1}^d \mu_i$ and $\nu = \prod_{i=1}^d \nu_i$ on $\mathbb{R}^d$ endowed with the euclidean metric:
\begin{align*} \dw(\mu,\nu)\le \sum_{i=1}^d \dw(\mu_i,\nu_i),\end{align*}
see \cite[Lemma 3]{mariucci2018}.  

We now present a  variation of this result to enable {the study of the convergence} of the Wasserstein distance between the empirical distribution of the energy minimizing configuration  ${\bf x}$ and the distribution specified by quantum theory.  Concentrating on the case $d=3$ for simplicity, the following result  provides a bound on the Wasserstein distance  between  two probability measures  $\mu $ and $\nu$ on $\mathbb{R}^3$ (of the type that is relevant here) in terms of the Wasserstein distance  between  the distributions  of their respective (signed) spherical coordinates (denoted $\mu_r$, $\mu_\theta$, $\mu_\phi$, etc).  

\begin{lemma}
\label{lem2.1}
Let $X\sim \mu $ and   $Y\sim \nu$ be random vectors in   $\mathbb{R}^3$.  If 
the signed spherical coordinates of $X$ are independent, and the same  holds for $Y$, then 
   \begin{align*}
\dw(\mu,\nu)\le  \dw(\mu_r,\nu_r)+  \sqrt{m_\mu m_\nu}[ \dw(\mu_\theta,\nu_\theta)+ \dw(\mu_\phi,\nu_\phi) ],
\end{align*}
where  $m_\mu =\int |x| \, d\mu_r(x)$ is the mean absolute deviation of the radial coordinate of $X$, similarly for $m_\nu$.
\end{lemma}

The result is the same in the case $d=2$, except the azimuthal contribution $\dw(\mu_\phi,\nu_\phi)$ is absent.  The result naturally extends to general $d\ge 3$, with additional terms of the  form $\dw(\mu_\theta,\nu_\theta)$ representing each of the  ($d-2$)  polar angles.  The directional contributions of the lowest energy configuration will be  shown to have a faster rate of convergence (as $N\to \infty$) than the radial contribution, so it will suffice to restrict attention to the convergence of the latter, as we do in the next section.

\begin{proof} 
For  points $\underline{x}_1, \underline{x}_2\in   \mathbb{R}^3$, the squared-euclidean distance between them in terms of their signed spherical coordinates $(r_i,\theta_i,\phi_i)$, $i=1,2$, is
\begin{align*}
  & \|\underline{x}_1 -\underline{x}_2\|^2 \nonumber\\
 &= r_1^2 +r_2^2 -2r_1r_2[\cos(\theta_1-\theta_2) + \sin(\theta_1)\sin(\theta_2)(\cos(\phi_1-\phi_2)-1)]\notag\\
&= (r_1-r_2)^2 -2r_1r_2[\cos(\theta_1-\theta_2)-1 + \sin(\theta_1)\sin(\theta_2)(\cos(\phi_1-\phi_2)-1)] \\
&\le  (r_1-r_2)^2+ r_1r_2[(\theta_1-\theta_2)^2 +(\phi_1-\phi_2)^2],\notag
\end{align*}
using the inequality $|\!\cos(x)-1|\le x^2/2$ for all $x\in  \mathbb{R}$.  Then using the inequality $(|a|+ |b| +|c|)^{1/2} \le |a|^{1/2} +|b|^{1/2} + |c|^{1/2}$ for any $a,b,c \in \mathbb{R}$, we obtain 
\begin{equation}
\label{polin}
\|\underline{x}_1 -\underline{x}_2\| \le |r_1-r_2|+\sqrt{ |r_1||r_2|} [|\theta_1-\theta_2| +|\phi_1-\phi_2|].
\end{equation}
We can choose coupled random vectors $X^\star\sim \mu$ and $Y^\star \sim \nu$ such that $$\dw(\mu_r,\nu_r)= \mathbb{E} |X_r^\star-Y_r^\star|,\ \ \dw(\mu_\theta,\nu_\theta)= \mathbb{E} |X_\theta^\star-Y_\theta^\star|,\ \ \dw(\mu_\phi,\nu_\phi)= \mathbb{E} |X_\phi^\star-Y_\phi^\star|,$$  
with the random vectors  $(X_r^\star,Y_r^\star)$, $(X_\theta^\star,Y_\theta^\star)$ and $(X_\phi^\star,Y_\phi^\star)$ being  independent.  Substituting 
$X^\star$ and $Y^\star$ for $\underline{x}_1$ and $\underline{x}_2$ in (\ref{polin}) and taking expectations of both sides leads to
$$\dw(\mu,\nu)\le \mathbb{E} \|X^\star-Y^\star\| \le \mathbb{E} |X_r^\star-Y_r^\star |+ [\mathbb{E} \sqrt{|X_r^\star| |Y_r^\star}|][ \mathbb{E} |X_\theta^\star-Y_\theta^\star|+ \mathbb{E} |X_\phi^\star-Y_\phi^\star| ].$$
The proof is completed using Cauchy--Schwarz to obtain $\mathbb{E} \sqrt{{|X_r^\star| |Y_r^\star}|}\le \sqrt{m_\mu m_\nu}$.  \end{proof}

\section{Stein's method for approximating radial distributions}

We begin this section with some generalities concerning the version of Stein's method developed in \cite{ernst2019distances,ernst2020first}.  We only give a summary here and refer the reader unfamiliar with Stein's method to the Supplementary material for more information and background. 

Consider a standardized (i.e.,\ mean 0, variance 1) random variable $F_\infty$ with density $ p_{\infty}$ (in the sequel  $p_{\infty}$ will take the form of a radial distribution), cdf $P_{\infty} = \int_{-\infty}^\cdot p_{\infty}$ and survival function $\bar{P}_{\infty} = 1 - P_{\infty}$.  To $p_{\infty}$ we associate  the ``density Stein operators'' \begin{align*}
   \mathcal{T}_{\infty}f(x)&  = \frac{(f(x) p_{\infty}(x))'}{p_{\infty}(x)}= f'(x) +  \frac{p_{\infty}'(x)}{p_{\infty}(x)} f(x) \\
  \mathcal{L}_{\infty}h(x) & 
                             = - \int_{-\infty}^{\infty}
   h'(y) \frac{P_{\infty}(y \wedge x) \bar{P}_{\infty}(y \lor x)}{p_{\infty}(x)} \mathrm{d}y
\end{align*}
acting  on absolutely continuous functions $f$ and $h$ such that, moreover, $f'$ and $h$ are integrable with respect to 
$p_{\infty}$.  One can show that $\mathcal{T}_{\infty} \mathcal{L}_{\infty}h(x)  = h(x)- \mathbb{E} h(F_{\infty})$, where $\mathcal{L}_{\infty} \mathcal{T}_{\infty}f(x)  = f(x)$ for all appropriate $h$ and $f$.
The choice $h(x) = -x$
 gives the  \emph{Stein kernel} 
\begin{align*}
  \tau_{\infty}(x) := \mathcal{L}_{\infty}h(x) = \int_{-\infty}^{\infty}
  \frac{P_{\infty}(y \wedge x) \bar{P}_{\infty}(y \lor x)}{p_{\infty}(x)} \mathrm{d}y.
\end{align*}
It is now well-documented that Stein kernels, if available,
provide valuable insight into the properties of the underlying
densities
.
Following \cite{ley2015distances,Do14}, we propose to apply Stein's method based on the  following  Stein operator: 
\begin{align*}
  \mathcal{A}_{\infty}g(x) = \mathcal{T}_{\infty} (\tau_{\infty}(x) g(x)) = \tau_{\infty}(x) g'(x) - x g(x).
\end{align*}
Note that, by design, $\mathbb{E} \left[ \mathcal{A}_{\infty} g (F_{\infty}) \right] = 0$ for all $g$ for which the moments exist.  The Stein equation for $p_{\infty}$ associated to operator $\mathcal{A}_{\infty}$ is then \begin{align}
   \tau_{\infty}(x) g'(x) - x g(x) = h(x) - \mathbb{E}h(F_{\infty}) \label{eq:stek2}
\end{align}
and one can easily check  that the function 
\begin{align}
  \label{eq:solstek2}
  g_h(x) =
  \frac{\mathcal{L}_{\infty} h(x)}{ \tau_{\infty}(x)}
\end{align}
solves \eqref{eq:stek2} at all $x$ such that $\tau_{\infty}(x) \neq 0$, so that for all real random variables  $F$ such that $ {P}(\tau_\infty(F) = 0) = 0$,   it holds
that
\begin{align}\label{eq:29}
  \mathbb{E} h(F) - \mathbb{E} h(F_{\infty})    & = \mathbb{E} \left[
    \tau_{\infty}(F) g_h'(F) - F g_h(F) \right]
\end{align}
with $g_h$ as given in \eqref{eq:solstek2}.  By definition of the Wasserstein distance, we conclude
\begin{equation}
  \label{eq:1-wassstinm}
  \dw(\mathbb{F}, \mathbb{F}_{\infty}) \le \sup_{h \in \mathcal{H}} \left| \mathbb{E} \left[
    \tau_{\infty}(F) g_h'(F) - F g_h(F) \right] \right|,
\end{equation}
where $\mathbb{F}$ is the probability distribution of $F$, $\mathbb{F}_{\infty}$ is that of $F_{\infty}$ and $\mathcal{H}$ is the family of Lipschitz functions with constant less than 1 (recall the first lines of Section \ref{sec:wass-dist-spher}).

Specializing to  densities of the form $p_{\infty}(x) = b(x) \varphi(x)$, with $\varphi(x)$ the standard normal density, leads to the following non-uniform bounds for  $g$; they are  based on  results from \cite{ernst2019distances} and proved  in  the supplementary material (see Lemma 2.4 therein).
\begin{proposition} \label{prop:boundsong} Let $p_{\infty}(x) = b(x) \varphi(x)$  be a standardized density on the real line, with $b(x)$ a positive function such that $b(x) \neq 0$ for all $x \neq 0$.  Define $$ R_{\infty}(x) = \frac{1}{{p_{\infty}(x)}}{ \int_{-\infty}^x{P}_{\infty}(u)
    \mathrm{d}u  \int_x^{\infty} \overline{P}_{\infty}(u)
    \mathrm{d}u}. $$
  Let $g (= g_h)$ be as in \eqref{eq:solstek2}, with $h \in \mathcal{H}$. Then \begin{align}
  & |g(x) | \le 1 \label{eq:boundg} \\
  & |g'(x)| \le 
   2 \frac{R_{\infty}(x)}{\tau_{\infty}(x)^2}=:
      \Psi_1(x) \label{eq:boundgp} \\
  & |g''(x)| \le
  \frac{2}{\tau_{\infty}(x)} \left( 1  + \left|
  \frac{2x}{\tau_{\infty}(x)} - x + \frac{b'(x)}{b(x)}  \right|
  \frac{R_{\infty}(x)}{
    \tau_{\infty}(x)}\right) =:
     2\Psi_2(x) \label{eq:boundgpp}
                                             \end{align}
                                           \end{proposition}
                                           An important aspect of inequalities
                                           \eqref{eq:boundg}, \eqref{eq:boundgp} and \eqref{eq:boundgpp}, is that they hold uniformly over all $h \in \mathcal{H}$; our approach to Stein's method consists in using this information in combination with \eqref{eq:1-wassstinm} in order to bound the Wasserstein distance.

                                           We now further specialize to the choice $b(x) \propto |x|^k$, where $k=d-1$ in the case of a $d$-dimensional radial distribution.  We then have the following explicit expression for the Stein kernel by  straightforward   integration of the expressions involved. 
                                           \begin{lemma}
  Let $p_{\infty}(x) \propto |x|^k \varphi(x)$ for given $k\ge 0$ with $\varphi(x)$ the standard Gaussian density. Then the Stein kernel 
  \begin{align}\label{eq:tauinfk}
&    \tau_{\infty}(x; k) = 2^{k/2} e^{x^2/2} |x|^{-k} \Gamma(1+k/2,
                   x^2/2),
  \end{align}
  where $\Gamma(\alpha, x) = \int_x^{\infty} t^{\alpha-1}e^{-t}
  \mathrm{d}t$ is the (upper) incomplete gamma function.
\end{lemma}

\begin{example}
The following results are immediate from \eqref{eq:tauinfk}: 
 \begin{itemize}

  \item if  $k=0$ (i.e., $p_{\infty}$ is standard Gaussian density) then
     $\tau_{\infty}(x) = 1$;   
  \item if  $k=1$ (i.e.,  $p_{\infty} $ is
     two-sided Rayleigh density) then $\tau_{\infty}(x) = 1 + 
    \sqrt{\pi/2} |x|^{-1} $ $ e^{x^2/2} \mathrm{Erfc}(x/\sqrt{2})$ (which
    behaves like $1/|x| + 1$);  
  \item if  $k=2$  (i.e.,    $p_{\infty}$ is
    two-sided Maxwell density) then $\tau_{\infty}(x) =
    2/x^2+1$.
  \end{itemize}
  It is not hard to obtain expressions for the Stein kernel at any level $k \in \mathbb{N}$; they are provided in the supplementary material (see Lemma~2.1 therein). 

 \end{example}

\begin{figure}
  \centering
    \includegraphics[width=0.45\textwidth]{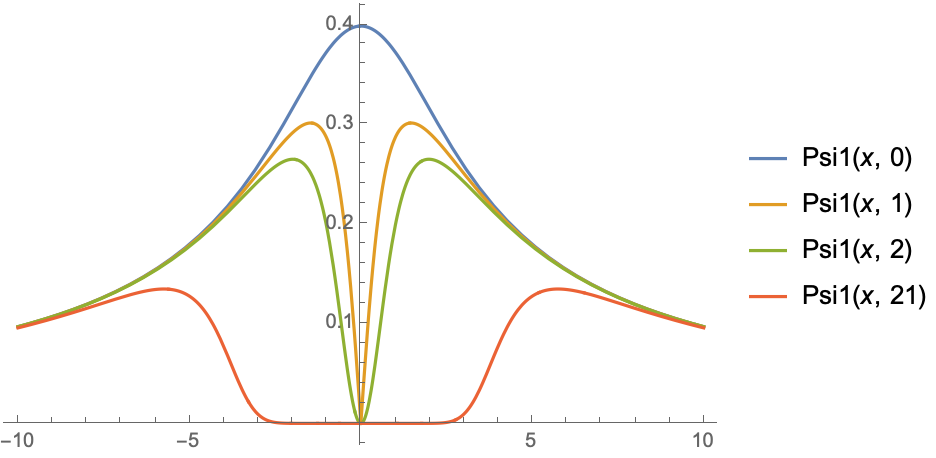} \qquad 
    \includegraphics[width=0.45\textwidth]{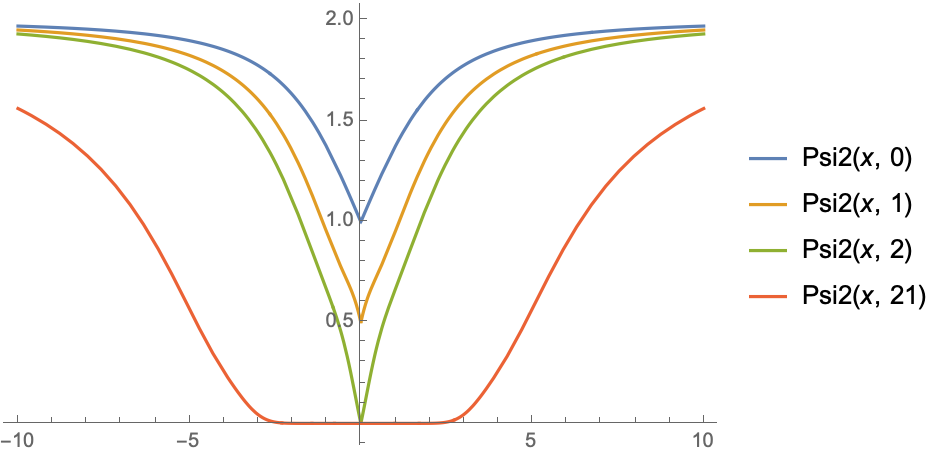}
    \caption{\small Functions $\Psi_1$ (left) and  $\Psi_2$ (right) defined in \eqref{eq:boundgp} and \eqref{eq:boundgpp} for $b(x) \propto  |x|^k$ with $k=0$ (blue),
       $k=1$ (orange), $k=2$ (green) and  $k=21$ (red) }
  \label{fig:psi12}
\end{figure}

The following properties of the functions defined in \eqref{eq:boundgp} and \eqref{eq:boundgpp} (with $\tau_{\infty}$ as given in \eqref{eq:tauinfk}) will be useful as well (see 
Fig.\ \ref{fig:psi12} for illustration\footnote{The supplementary material contains more information on these functions.}):
$\Psi_1$ and $\Psi_2$  are
  even and positive, and 
  
  \begin{itemize}
  \item if $k=0$ then   $\Psi_1(x)$ is unimodal with maximum
    $\Psi_1(0) = \sqrt{2/\pi}$ at $x=0$ and
strictly decreasing towards 0 for $x\ge 0$; $\Psi_2(x)$ is unimodal
with minimum $\Psi_2(0) = 1$ at $x=0$ and strictly increasing towards 2  for $x\ge 0$.

  \item If $k=1$ then   $\Psi_1(x)$ is bimodal with minimum 
    $\Psi_1(0) = 0$ at $x=0$, maximum value less than 1  after which
    it is 
strictly decreasing towards 0 for $x\ge 0$; $\Psi_2(x)$ is unimodal 
with minimum $\Psi_2(0) = 1/2$ at $x=0$ and strictly increasing towards 2  for $x\ge 0$.

  \item If  $k \ge 2$ then   $\Psi_1(x)$ is bimodal with minimum
    $\Psi_1(0) = 0$ at $x=0$, maximum value less than 1  after which
    it is strictly decreasing towards 0 for $x\ge 0$; $\Psi_2(x)$ is unimodal 
    with minimum $\Psi_2(0) = 0$ at $x=0$ and strictly increasing towards 2  for $x\ge 0$.

\end{itemize}

Having presented the key elements of Stein's method for radial distributions, we now give some properties of the corresponding approximating sequences. The following lemma, ensuring a unique symmetric and monotone solution to \eqref{genrec}, is a routine extension of a result in \cite{mckeague2016,MPS19} to general $k$.  In the sequel we assume $N$ is even and define the median of an ordered set $\{x_1,\ldots, x_N\}$ to be $(x_m+x_{m+1})/2$, where $m=N/2$.

\begin{lemma}\label{lem1} 
Every zero-median solution $\{x_1,\ldots, x_N\}$ of  \eqref{genrec} for a given $k\ge 0$ satisfies:
\begin{itemize}
\item[]
\begin{itemize}
\item[{\rm(P1)}] Zero-mean: $x_1+\ldots +x_N=0$.
\item[{\rm (P2)}] Variance-bound:   $x_1^2+\ldots +x_N^2=(k+1)(N-1)$.
\item[{\rm(P3)}]  Symmetry:  $x_n=-x_{N+1-n}$ for $n=1,\ldots, N$.
\end{itemize}
\end{itemize}
Further, there exists a unique solution $\{x_1,\ldots, x_N\}$ such that (P1) and
\begin{itemize}
\item[]
\begin{itemize}
\item[{\rm(P4)}]  Strictly decreasing:  $x_1>\ldots >x_N$
\end{itemize}
\end{itemize}
hold.  This solution has the zero-median property, and thus also satisfies (P2) and (P3).
\end{lemma}
Let ${\mathbb F}_N$ be the empirical distribution of the unique solution $ \left\{ x_1, \ldots, x_N \right\}$. 
For the case $k=0$, the following result gives the optimal rate of convergence in Wasserstein distance of ${\mathbb F}_N$ to standard Gaussian.  { This result is not new, as the upper bound follows by adapting the coupling argument in \cite{MPS19}, and the lower bound, which follows from a very delicate analysis of the $x_i$'s, was recently proved in \cite{chen2020optimal}.  We } provide a new proof for the upper bound in Example \ref{egk0} below.

\begin{proposition}\label{prop:boundk0}    The {
  empirical distribution $\mathbb{F}_N$} on the unique symmetric and monotone solution to \eqref{e1} (or \eqref{genrec} with $k=0$) satisfies 
  \begin{align*}
    \dw({\mathbb F}_N,{\mathbb F}_{\infty}) \asymp \sqrt{\log N}/ N,
 \end{align*}
 where ${\mathbb F}_{\infty} $ 
  is { the standard Gaussian distribution}.
\end{proposition}

Upper bounds on the rate of convergence for values of $k\ge 1$ can also be obtained by adapting the arguments in \cite{MPS19} using the coupling approach (as developed in detail in the Supplementary material).  The following result for the case $k=1$ is based on this approach (and is detailed in Section~2.3 of the Supplementary material), but it is not expected to yield the optimal rate which, as argued in the sequel, we expect to be the same as the Gaussian case in the above proposition.

\begin{proposition} \label{prop:boundk1} The {
  empirical distribution $\mathbb{F}_N$} on the  unique symmetric and monotone solution to (\ref{Ralrec})  (or \eqref{genrec} with $k=1$) satisfies 
  \begin{align*}
  \dw({\mathbb F}_N,{\mathbb F}_{\infty})= O({\log N}/ N),
\end{align*}
where ${\mathbb F}_{\infty} $ is the two-sided Rayleigh distribution.
\end{proposition}

The main result of this section is the following theorem based on a discrete density approach to Stein's method, cf.\ the approach of \cite{goldstein2013stein} in the case of points on an integer grid. This result is entirely new and may  be of independent interest as the upper bound it provides  does not rely on any specific structure of  $  x_1, \ldots, x_N $ apart from symmetry and monotonicity, and therefore may be much more widely applicable.  It leads to the new  proof of the tight upper bound in Proposition~\ref{prop:boundk0} for the case $k=0$, as shown in Example~\ref{egk0}  below.   For $k\ge 1$, it becomes more difficult analytically to determine the order of the  bound, but it nevertheless provides a tight bound that can be evaluated numerically and that agrees with the asymptotic order of $\dw({\mathbb F}_N, {\mathbb F}_\infty)$ as $N\to \infty$, as discussed in Example~\ref{egk12}  below.

\begin{theorem}\label{thm:approxsequ}
  Let $N$ be an even integer and  $ x_1, \ldots, x_N $ be any symmetric and strictly monotone decreasing sequence. {Set $\tau_N(x_i) =  (x_{i} - x_{i+1})  \sum_{j=1}^ix_j$ for $1 \le i \le N-1$, and $\tau_N(x_N) = 0$}.  Let ${\mathbb F}_N $ be the empirical distribution of these points and let {${\mathbb F}_{\infty}$ be the probability measure with density} $p_{\infty}$ satisfying the assumptions of Proposition \ref{prop:boundsong}. Then
  \begin{align}
\dw({\mathbb F}_N, {\mathbb F}_{\infty})  
  & \le \label{eq:hope2}
    \frac{1}{N}
                                                    \sum_{i=1}^{N} 
\left| \tau_{\infty}(x_i) {-  \tau_N(x_i)} \right|
    \Psi_1(x_i)\\
  & \label{eq:hope3} 
\quad              + \frac{1}{N} \sum_{i=1}^{N-1}  { |x_{i} - x_{i+1}|\tau_N(x_i)}
                                                                            \max\{\Psi_2(x_i), \Psi_2(x_{i+1})
    \},
\end{align}
where $\Psi_1$ and $\Psi_2$ are  given in \eqref{eq:boundgp} and \eqref{eq:boundgpp}, respectively. 
\end{theorem}

\begin{remark}
  Equation \eqref{eq:hope2} encourages us to consider the quantity  $ \tau_N(x_i) =  (x_{i} - x_{i+1}) \sum_{j=1}^ix_j$ as a Stein kernel for the uniform distribution on the set $\left\{ x_1, \ldots, x_N \right\}$.   
This function is non-negative and, under the conditions that we are considering, it holds that $\tau_N(x) \approx \tau_{\infty}(x)$ throughout the support of ${\mathbb F}_N$  except close to the edges and  the origin (see Figs.\   \ref{fig:prettytau} and \ref{fig:prettytau2}). 
 
\end{remark}

\begin{figure}[!ht]
\begin{center}
 \includegraphics[scale=.3]{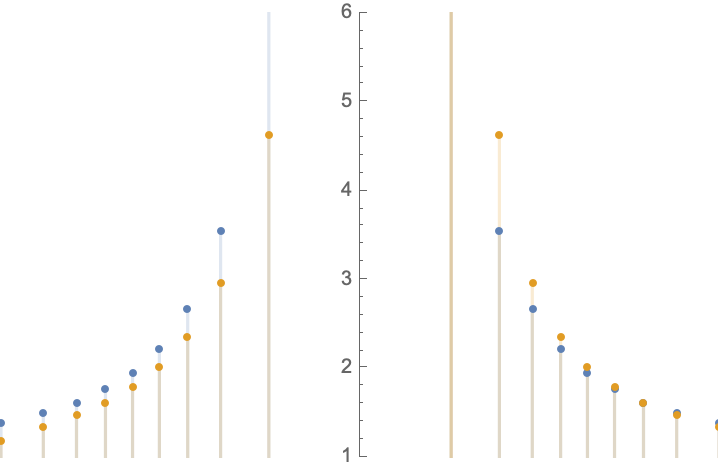} \includegraphics[scale=.3]{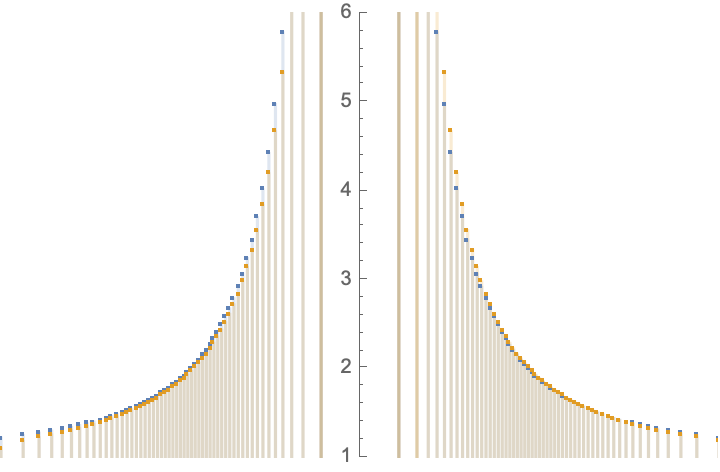}\includegraphics[scale=.3]{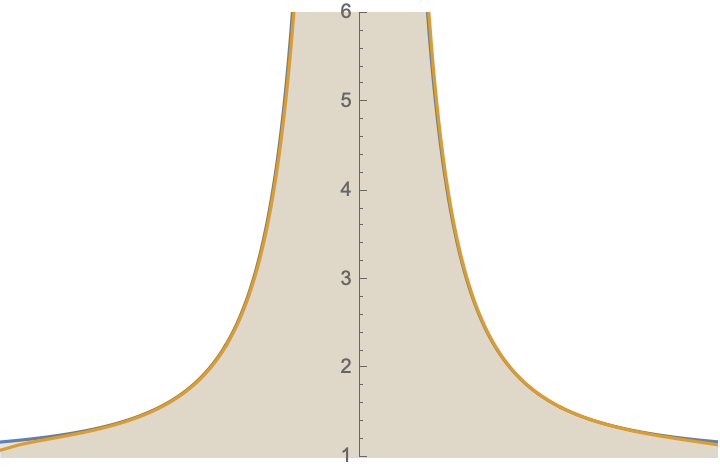}
 \caption{ Superimposition of $\tau_N(x_i), 1 \le i \le N$ (orange bars) and $\tau_{\infty}(x_i), 1 \le i \le N$ (blue bars)  when  $\tau_{\infty}(x) = 1+2/x^2$ is the Stein kernel for the two-sided Maxwell distribution and  $(x_i)_{1 \le i \le N}$ is as in Lemma \ref{lem1} with  $k=2$ and $N = 20$ (left), $N = 110$ (middle) and $N = 300$ (right)  
 }
   \label{fig:prettytau}
 \end{center}
\end{figure}

\begin{figure}[!ht]
\begin{center}
 \includegraphics[scale=.3]{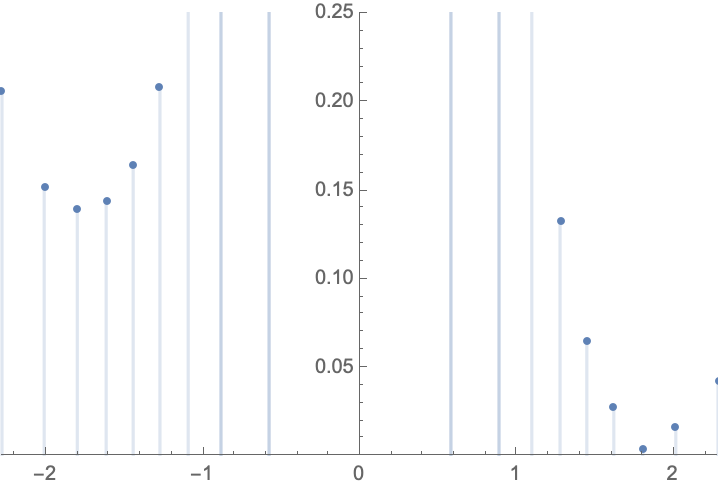} \includegraphics[scale=.3]{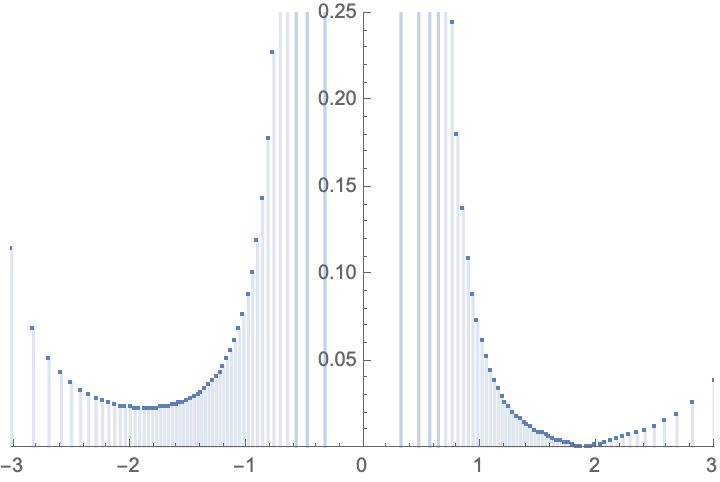}\includegraphics[scale=.3]{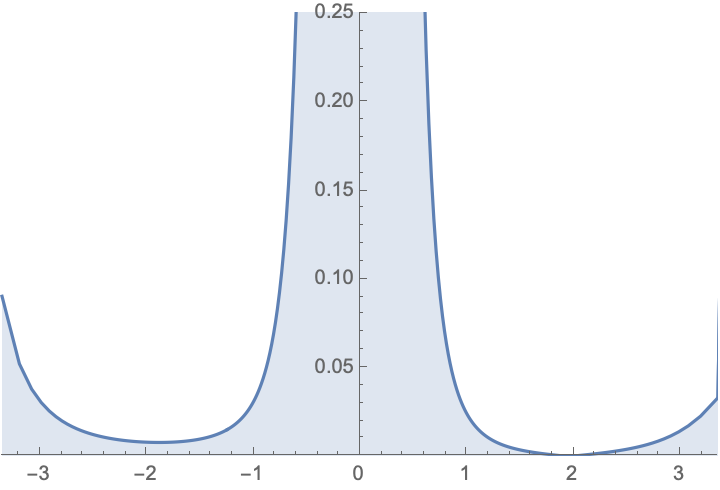}
 \caption{ Plot of   $|\tau_N(x_i) - \tau_{\infty}(x_i)|, 1 \le i \le N$  when  $\tau_{\infty}(x) = 1+2/x^2$ is the Stein kernel for the two-sided Maxwell distribution  and   $(x_i)_{1 \le i \le N}$ is as in Lemma~\ref{lem1} with  $k=2$ and $N = 20$ (left), $N = 110$ (middle) and $N = 300$ (right)  
 }
   \label{fig:prettytau2}
 \end{center}
\end{figure}

\begin{remark}
  Since both $\Psi_1$ and $\Psi_2$ are uniformly bounded by 2, we immediately  obtain  the simpler upper bound
  \begin{align*}
\dw({\mathbb F}_N, {\mathbb F}_{\infty})  
  & \le 
    \frac{1}{N}
                                                    \sum_{i=1}^{N} 
\left| \tau_{\infty}(x_i)  - \tau_N(x_i)  \right|
     + \frac{1}{N} \sum_{i=1}^{N-1}   |x_i-x_{i+1}|\tau_N(x_i).
  \end{align*}
  This simplicity comes at the cost of a loss of precision in the rate. For instance, in the case that we consider here, the presence of the functions $\Psi_1$ and $\Psi_2$ in the bounds controls the problems around the origin.  \end{remark}

\begin{example}[Case $k=0$]\label{egk0} The sequence $\{x_n\}$
  satisfies $x_n-x_{n+1}  =  ({\sum_{i=1}^n
    x_i})^{-1}$ so that $\tau_N(x_i) = 1$ for all $i = 1, \ldots, N-1$. Since $ \tau_\infty(x_i)=1$,
  \eqref{eq:hope2} vanishes except for the term corresponding to $i = N$ leading to $\Psi_1(x_N)/N \le 1/N$. Also,  \eqref{eq:hope3}  reads
  \begin{align*}
&        \frac{1}{N} \sum_{i=1}^{N-1}   |x_{i+1} - x_i|
                                                                            \max\{\Psi_2(x_i),
                                 \Psi_2(x_{i+1})\}.
  \end{align*}
  Using $\Psi_2 \le 2$ along with the symmetry
  of the sequence, we then have 
  \begin{align*}
\dw({\mathbb F}_N, {\mathbb F}_{\infty}) \le \frac{1}{N}+ \frac{2}{N} \sum_{i=1}^{N-1}   |x_{i+1} - x_i| = \frac{1+4x_1}{N}= O(\sqrt{\log N}/N),
\end{align*}
   where the last step follows by a version of the argument in   \cite[Lemma 4.8]{MPS19} showing that  $x_1=O(\sqrt{\log N})$.  This proves the  upper bound  part  of Proposition \ref{prop:boundk0}. 
\end{example}
         
\begin{example}[Cases $k\ge 1$]\label{egk12} Evaluating the bound in Theorem \ref{thm:approxsequ} based on numerical solutions to \eqref{genrec} and comparing the results with values of $\dw(\mathbb{F}_N, {\mathbb F}_{\infty})$ calculated using numerical integration, over a range of values of $N$, suggests that the  bound is of the same order as $\dw(\mathbb{F}_N, {\mathbb F}_{\infty})$, and moreover that
  \begin{align*}
&    \dw({\mathbb F}_N, {\mathbb F}_{\infty})\asymp \sqrt{\log N}/ N \mbox{ for } k=1,2, \\
 &  \dw({\mathbb F}_N, {\mathbb F}_{\infty})\asymp ({\log N})^6/ N^2 \mbox{ for } k\ge 3.
  \end{align*}
  To obtain these rates we used \textsf{Mathematica} to compute the  bounds and \textsf{R} to compute the Wasserstein distance using its  representation as the $L_1$-distance between the cdfs (the programs are provided in the Supplementary material).

\end{example}

\begin{remark}
  We emphasize that the bound in Theorem \ref{thm:approxsequ} holds irrespective of the sequence that is chosen;  it may be of interest to optimize the approximation of $p_{\infty}$ by a sample $x_1< \ldots < x_N$ which minimizes the right hand side. One simple way to do this is to require that the sequence satisfies the recurrence
  \begin{equation*}
   x_{i+1} = x_i - \frac{\tau_{\infty}(x_i)}{ \sum_{j=1}^i x_j},
 \end{equation*}
 thereby canceling out \eqref{eq:hope2} and only leaving \eqref{eq:hope3}; note that in the Gaussian case $\tau_{\infty} = 1$ so  we are back with the recurrence \eqref{e1}. 
\end{remark}
\begin{proof}[Proof of Theorem \ref{thm:approxsequ}]
  The idea is to apply a  version of Stein's density method to $F_N \sim {\mathbb F}_N$. Note that a  discrete ``derivative''  at  $x_i$, $i= 1, \ldots, N-1$, is given by \begin{equation*} D_Nf(x_i) = f(x_{i}) - f(x_{i+1}).  \end{equation*}
  Writing {$ \sigma(x_{i}) =  \sum_{j=1}^{i} x_{j}$}, straightforward summation (along with the identity  $x_N = - \sum_{j=1}^{N-1}x_j$) shows  \begin{equation*} \mathbb{E} \left[ \sigma(F_N) D_N g(F_N) - F_N g(F_N) \right] = 0 \end{equation*} for all summable functions $g$.  Subtracting this from \eqref{eq:29} applied to $F = F_N$ gives \begin{align*}
  \mathbb{E} h(F_N) - \mathbb{E} h(F_{\infty})  & = \mathbb{E} \left[
                                                  \tau_{\infty}(F_N)
                                                  g'(F_N) - F_N g(F_N)
                                                  \right] \nonumber \\
                                                &= \mathbb{E} \left[  \tau_{\infty}(F_N) g'(F_N)  -  \sigma(F_N)
                                                  D_Ng(F_N)  \right].
                                                                                                                                                                                                                                                                                                                                                           \end{align*}
                                                                                                                                                                                                                                                                                                                                             {              From \eqref{eq:1-wassstinm} it then follows that
                                                                                                                                                                                                                                                                                                                                                           \begin{equation*}
                                                                                                                                                                                                                                                                                                                                                             \dw( \mathbb{F}_N, \mathbb{F}_{\infty}) 
                                                       \le \sup_{h \in \mathcal{H}}  \left|  \mathbb{E} \left[  \tau_{\infty}(F_N) g'(F_N)  -  \sigma(F_N)
                                                         D_Ng(F_N)  \right]                                                                                   \right|                                                                                                                                                                                                               \end{equation*}
                                                     with $g (= g_h)$ satisfying the bounds in Proposition \ref{prop:boundsong}.}
By Taylor's theorem, \begin{align*}
  D_Nf(x_i) =  {(x_{i}-x_{i+1}) f'(x_i) -}  {\small {1\over 2}}
{(x_{i}-x_{i+1})^2}
f''(c_i) {\mbox{ for all } 1 \le i \le N-1,}
\end{align*}
where $c_i$ is between  $x_{i+1}$ and $x_i$. 
Hence
\begin{align*}
&     \mathbb{E} h(F_N) - \mathbb{E} h(F_{\infty})  \\ & = \frac{1}{N}
                                                    \sum_{i=1}^{{N-1} }
\left( \tau_{\infty}(x_i) -  {\sigma(x_i) } { (x_{i} - x_{i+1})} \right)
                                                         g'(x_i) { {-}
                                                  \frac{1}{2N}
                                                         \sum_{i=1}^{{ N-1} }
                                                          {\sigma(x_i)}
                                                         ({x_{i} - x_{i+1}} )^2
    g''(c_i) }\\
  & \qquad +  \tau_{\infty}(x_N) g'(x_N)   \\
  & = \frac{1}{N}
                                                    \sum_{i=1}^{{N} }
\left( \tau_{\infty}(x_i)  - \tau_N(x_i) \right)
                                                         g'(x_i) - 
                                                  \frac{1}{2N}
    \sum_{i=1}^{{N-1} }  (x_{i+1} - x_i) \tau_N(x_i) 
    g''(c_i)
\end{align*}
(recall that $\tau_N(x_N) = 0$). 
Taking absolute values, using
\eqref{eq:boundgp} and \eqref{eq:boundgpp}, along with the  symmetry/unimodality of the function $\Psi_2$,  we get the result.
\end{proof}



\section{Convergence of multidimensional ground states}

In this section we apply the results of the previous section to the $N$-point empirical distributions $\mathbb{P}_{N}$ of the $d$-dimensional ground states discussed in Sections \ref{2D} and \ref{3D}.  

These states  have (signed) radial components  given by the unique  zero-median solution $x_1>\ldots >x_R$ to the recursion  \eqref{genrec} with $b(x)= |x|^k$, where  $k=d-1 \ge 1$ and $R=N^{1/2-1/(2d)}$ plays the role of  $N$ in \eqref{genrec}, corresponding to $R$  points in each of $N^{1/2 +1/(2d)}$ radial directions.  By translating the results in Example \ref{egk12} into the $d$-dimensional setting,  the optimal convergence rates of the radial part of $\mathbb{P}_{N}$ to the radial part of the $d$-dimensional normal distribution ${\cal N}_d$ are seen to be given by

\m

 $d=2,3$:  \ \ \ $\dw(\mathbb{P}_N^{\rm radial}, {\cal N}_d^{\rm radial})\asymp \sqrt{\log R}/ R$;
\m

 $d\ge 4$: \ \ \ \ \ \  $\dw(\mathbb{P}_N^{\rm radial}, {\cal N}_d^{\rm radial} )\asymp ({\log R})^{6}/ R^2$.
\m

The number of distinct polar angles in the directional part of the ground state (each with $d-2$ components taking values in $[0, \pi/2]$ when $d\ge 3$, and a single component in $[0,\pi)$ when $d=2$)  is of order  $ N^{1/d}$.   For $d\ge 3$, a larger number (namely $N^{1/2 -1/(2d)}$) of azimuthal angles corresponding to each polar angle are available, so the polar part has the slower convergence rate.  
In general, the Wasserstein distance between the directional part of $\mathbb{P}_{N}$ and that of ${\cal N}_d$ is thus of order  $ N^{-1/d}$ (this follows easily using the  representation of  Wasserstein distance as the $L_1$-distance between  two cdfs).  For $d\ge 4$, in view of the upper bound in  Lemma \ref{lem2.1}, namely  a sum involving separate contributions from the the radial part and the directional parts, we conclude that the {\it  overall   rate} that $\mathbb{P}_{N}$ tends to $  {\cal N}_d$ is of order $ N^{-1/d}$.  For $d=3$, the radial component involving $N^{1/3}$ points dominates, however, so  the overall rate is of order $\sqrt{\log N}/N^{1/3}$.  Similarly, the radial component involving $N^{1/4}$ points dominates in the case  $d=2$, so  the overall rate is of order $\sqrt{\log N}/N^{1/4}$. 

 In summary, we obtain
\m

 $d=2$:  \ \ \ $\dw(\mathbb{P}_N, {\cal N}_d )\asymp \sqrt{\log N}/N^{1/4}$;
 
 \m
 $d=3$:  \ \ \ $\dw(\mathbb{P}_N, {\cal N}_d )\asymp \sqrt{\log N}/N^{1/3}$;
\m

 $d\ge 4$: \ \ \   $\dw(\mathbb{P}_N, {\cal N}_d )\asymp N^{-1/d}$.
\m

\noindent Interestingly, the fastest rate of  convergence to the ground state occurs in three dimensions.  

For the excited states discussed in Section \ref{excited}, the rates of convergence are of order $N^{-1/d}$; in this case the {\it directional} components  dominate because the radial component has the faster  rate (as seen by applying Example \ref{egk12} with $k=d+1$  for $d=2,3$).


\appendix

\pagebreak

\section{Supplementary material}

All code for performing computations is available at the url
\begin{center}
https://tinyurl.com/mcks2021code    
\end{center}
The code is also available at the end of the supplementary material. 
\subsection{Preliminary remarks on Stein operators}
\label{sec:steins-dens-appr}

Consider a probability distribution $\mathbb{P}$ with cdf $P$, pdf $p$
w.r.t.\ the Lebesgue measure on $\mathbb{R}$. Suppose that $p$ itself
is absolutely continuous, with a.e.\ derivative $p'(x)$.  We denote
$ L^1(p)$ the collection of functions such that
$\int_{-\infty}^{\infty}|h(x)| p(x) \mathrm{d}x < \infty$ and write
$P(h) = \mathbb{E}_ph = \int_{-\infty}^{\infty}h(x) p(x) \mathrm{d}x.$
We also denote $\mathcal{F}^{(0)}(p)$ the collection of all mean 0
functions under $p$.  Following \cite{ernst2020first}, to $p$ we associate
the Stein operators \begin{align}
  \label{eq:10}
  & \mathcal{T}_pf(x) = \frac{(f(x)p(x))'}{p(x)}\\
\label{eq:11}  & \mathcal{L}_ph(x) = \frac{1}{p(x)} \int_{-\infty}^x
                 (h(u) - P(h)) p(u) \mathrm{d}u
\end{align}
with the convention that $\mathcal{T}_pf(x) = \mathcal{L}_ph(x) = 0$
for all $x$ such that $p(x)=0$. In the sequel we denote
$\mathcal{S}(p) = \left\{ x \, | \, p(x)>0\right\}$,
$a = \inf \mathcal{S}(p)$ and $b = \sup \mathcal{S}(p)$; we assume
that $\mathcal{S}(p)$ is the union of a finite number of intervals.

Of course \eqref{eq:10} is only defined for functions $f$ such that $fp$ is absolutely continuous. We denote $\mathcal{F}(p)$ the collection of functions $f$ such that not only is $fp$ absolutely continuous, but also $(f(x) p(x))' \in L^1(p)$ and $\lim_{x \to a}f(x) p(x) = \lim_{x \to b}f(x) p(x) =0$.  This class of functions is important because $\mathcal{T}_pf \in \mathcal{F}^{(0)}(p) $ for all $f \in \mathcal{F}(p)$; this is crucial for Stein's method as it gives rise to many ``Stein identities'' which can be used for a variety of purposes.  Similarly, \eqref{eq:11} is only defined for functions $h \in L^1(p)$, in which case $\mathcal{L}_ph \in \mathcal{F}(p)$ for all $h \in L^1(p)$.

As described in \cite{ley2015distances}, it is interesting to ``standardize'' the
operator \eqref{eq:10} by fixing some $c \in \mathcal{F}(p)$ and
considering the family of ``standardized Stein operators''
$\mathcal{A}_c f(x) = \mathcal{T}_p(cf)(x) = c(x)f'(x) + \Big(c'(x) +
{p'(x)}/{p(x)} c(x)\Big) f(x)$ acting on some class
$\mathcal{F}(\mathcal{A}_c)$ made of all functions for which
$cf \in \mathcal{F}(p)$. Note that, by definition, we have
$\mathbb{E}[ \mathcal{A}_cf(X)] = 0$ for all
$f \in \mathcal{F}(\mathcal{A}_c)$.  It is important that $c$ be
well-chosen to ensure that $\mathcal{A}_c$ has a manageable expression;
as is now well known, there are many instances of densities $p$ (even
intractable densities) for which this turns out to be possible,
leading to many powerful handles on $p$ which can then serve for a
variety of purposes including but not limited to distributional
approximation.

Given an operator $\mathcal{A}_c$, classical instantiations of Stein's
method begin with a ``Stein equation'', i.e.\ a differential equation
of the form
  \begin{equation}
    \mathcal{A}_cf(x) =h(x)- P(h)\label{eq:steingn}
  \end{equation}
  for $h$ some function belonging to a class $\mathcal{H}$ of test
  functions. Typically, Stein's method practitionners work with one of
  the following classes: (i)
  $h \in \mathcal{H} := \mathrm{Kol}(\mathbb{R}) = \left\{
    \mathbb{I}(-\infty, z], z \in \mathbb{R} \right\}$ the indicators
  of a lower half line; (ii)
  $h \in \mathcal{H} :=\mathrm{TV}(\mathbb{R}) $ the collection of
  functions such that $\|h \| \le 1$; (iii)
  $h \in \mathcal{H} := \mathrm{Wass}(\mathbb{R})$ the collection of
  Lipschitz functions such that $\|h'\| \le 1$. In the sequel, we
  restrict our attention to $\mathcal{H} = \mathrm{Wass}(\mathbb{R})$,
  and we assume that for each $h \in \mathcal{H}$ there exists a
  unique function $f \in \mathcal{F}(\mathcal{A}_c)$ for which
  \eqref{eq:steingn} holds for all $x \in \mathcal{S}(p)$. Under
  ``reasonable assumptions on $p$'' (to be verified on a case-by-case
  basis) we can write
  $\mathcal{T}_p \mathcal{L}_p h = h - \mathbb{E}_ph$ for all
  $h \in L^1(p)$ and in particular
  $\mathcal{T}_p \mathcal{L}_p = \mathrm{Id}$ over
  $\mathcal{F}^{(0)}(p)$ ($\mathrm{Id}$ is the identity function). Similarly
  $\mathcal{L}_p \mathcal{T}_p = \mathrm{Id}$ over
  $\mathcal{F}(p)$. In other words, under ``reasonable assumptions on
  $p$'', the solution to \eqref{eq:steingn} is
  $ f_h(x) = {\mathcal{L}_ph(x)}/{c(x)}$ at all $x \in \mathcal{S}(p)$
  for which $c(x) \neq 0$.  Then, since the Wasserstein distance
  between two probability measures $\mathbb{P}$ and $\mathbb{Q}$ can
  be written as 
  $\dw(\mathbb{P}, \mathbb{Q}) = \sup_{h \in \mathrm{Wass}(\mathbb{R})} |\mathbb{E}h(X)
  -\mathbb{E}h(Y)|$ where $X \sim \mathbb{P}$ and $Y \sim \mathbb{Q}$, it holds that
  \begin{equation}
    \label{eq:22}
\dw(\mathbb{P}, \mathbb{Q}) = \sup_{h \in \mathrm{Wass}(1)}\left| \mathbb{E} \left[
  \mathcal{A}_c f_h(Y) \right] \right|;
  \end{equation}
   Stein's method in Wasserstein distance consists in exploiting this last identity for the
   purpose of estimating the Wasserstein distance  
   between the laws  $\mathbb{P}$ and $\mathbb{Q}$.

   In order to be able to use \eqref{eq:22} successfully, it is
   crucial to control solutions $f_h$ and their derivatives. In
   \cite{ernst2019distances} the following representations for
   \eqref{eq:11} are provided (recall that $h$ is Lipschitz with a.e.\ derivative $h'$): \begin{align} -
     \mathcal{L}_ph(x) & =- \mathbb{E} \left[ ({h}(X)- P(h))
                         \frac{\mathbb{I}[X \le x]  }{p(x)} \right] \nonumber\\
                                                        & = -
                                                          \mathbb{E}
                                                          \left[
                                                          ({h}(X) - P(h)) \frac{\mathbb{I}[X \le x] -P(x)}{p(x)}\right] \nonumber\\
                                                        & = \mathbb{E}
                                                          \left[
                                                          (h(X_2) -
                                                          h(X_1))\frac{
                                                          \mathbb{I}[X_1
                                                          \le x \le X_2] }{p(x)}\right] \label{eq:13}\\
                                                        & = \mathbb{E}
                                                          \left[ h'(X)
                                                          \frac{P(x
                                                          \wedge
                                                          X)(1-P(x
                                                          \lor
                                                          X))}{p(x)p(X)}
                                                          \right]\label{eq:14}
\end{align}
where, in \eqref{eq:13}, the random variables $X_1, X_2$ are
independent copies of $X$. A simpler way to write \eqref{eq:14} is
\begin{align*}
-\mathcal{L}_ph(x)  =  \int_{-\infty}^{\infty}
     h'(y) \frac{P_{\infty}(y \wedge x) \bar{P}_{\infty}(y \lor x)}{p_{\infty}(x)} \mathrm{d}y
\end{align*}
It is also shown that 
\begin{align*}
  \bar{h}(x) := h(x) - P(h) = \mathbb{E} \left[ h'(X)
  \frac{P(X) - \mathbb{I}[x \le X]}{p(X)} \right]
\end{align*}
for all $x \in \mathcal{S}(p)$.  All these representations will be
used in the next section to control the solutions to the Stein
equations; this in turn will lead to the distributional approximation
results.

\subsection{Stein's method for radial distributions }
\label{sec:generalization}

\subsubsection{Notations and background}
\label{sec:notations-background}

Before specializing to radial densities, it is enlightening to
first widen the scope somewhat and consider targets $F_{\infty}$ with
density of the form $p_{\infty}(x) = b(x) \gamma(x)$, for $\gamma$
some ``basis density'' and $b$ some positive $\gamma$-integrable
``tilting'' function.  This theory may also be of independent
interest.

First note that, in order for $p_{\infty}$ to be a density, it is necessary that $b\ge 0\in L^1(\gamma)$ and $\mathbb{E}[b(Z)] = 1$, where here and throughout we denote $Z \sim \gamma$. We further impose the following assumptions on $p_{\infty}$. First, we require that $\gamma$ is a differentiable probability density function with support the full real line, such that moreover $\gamma' \in L^1(\mathrm{d}x)$ has exactly one sign change (which, for simplicity, we fix at 0) and $\int \gamma'(x) \mathrm{d}x = 0$. Second, we let $B(x)$ be an absolutely continuous nondecreasing function with continuous derivative $b$, we denote $\mathcal{S}_B = \left\{ x \in \mathbb{R}\right.$, such that $\left.b(x)>0 \right\}$ and suppose that $\mathcal{S}_B$ is the union of a finite number of intervals. Following \cite{ley2015distances}, we also introduce $\mathcal{F}(\gamma)$ the Stein class of $\gamma$; this is the class of functions $f : \mathbb{R} \to \mathbb{R}$ such that $(f\gamma)' \in L^1(\mathrm{d}x)$ and $\int_{\mathbb{R}} (f(x)\gamma(x))' \mathrm{d}x = 0$.  We assume that $B \in \mathcal{F}(\gamma)$; since $b \in L^1(\gamma)$, this assures us that $\int b \gamma = - \int B \gamma'$ so that integration by parts holds without a remainder term.  Finally, letting $F_{\infty} \sim p_{\infty}$, we impose that $\mathbb{E} F_{\infty} ( = \mathbb{E}[Z b(Z)]) = 0$.

With these assumptions we are now ready to provide a Stein's method theory for $p_{\infty} = b\gamma$; the backbone of our approach comes from \cite{MPS19}.

\begin{definition}[Generalized $(b,\gamma)$-bias transformation]\label{def:genbgamtran}
  Suppose that $F$ is such that $P(F\in \mathcal{S}_{{B}}) = 1$ and define
  \begin{align*}
    \sigma^2_B(F) = \mathbb{E} \left[- \frac{\gamma'(F)}{\gamma(F)}
  \frac{B(F)}{b(F)}   \right].
\end{align*}
 The
random variable $F^{\star}$ satisfying
  \begin{align*}
      \sigma^2_B(F)   \mathbb{E} \left[
      \frac{f'(F^{\star})}{b(F^{\star})} \right] = \mathbb{E} \left[ - \frac{\gamma'(F)}{\gamma(F)}
      \frac{f(F)}{b(F)} \right] 
\end{align*}
for all $f$ such that both integrals exist is said to have the
generalized $(b,\gamma)$-bias distribution. The random variable
$F^{\star}$ is the generalized $(b,\gamma)$-bias transform of $F$.
\end{definition}
By construction, we always have
    \begin{align*}
      \sigma^2_B(F_{\infty})  & =  \mathbb{E} \left[- \frac{\gamma'(F_{\infty})}{\gamma(F_{\infty})}
                                \frac{B(F_{\infty})}{b(F_{\infty})}   \right]   
                              = -
                             \int_{\mathcal{S}_B} \gamma'(x)B(x)
                             \mathrm{d}x = \int_{\mathcal{S}_B} b(x) \gamma(x) 
                             \mathrm{d}x = \mathbb{E} b(Z) =1.
    \end{align*}
    Moreover, for any sufficiently regular function $f$: 
    \begin{align*}
      \mathbb{E} \left[- \frac{\gamma'(F_{\infty})}{\gamma(F_{\infty})}
  \frac{f(F_{\infty})}{b(F_{\infty})}   \right]  =  \mathbb{E} \left[ \frac{f'(F_{\infty})}{b(F_{\infty})} \right]. 
    \end{align*}
    Therefore $F_{\infty} = F^{\star}$, i.e.\ $p_{\infty}$ is a fixed
    point of the generalized $(b,\gamma)$-bias transform.  More generally, the following holds true.

    \begin{lemma} If $\sigma^2_B(F) <\infty$ then its generalized
  $(b,\gamma)$-bias transform $F^{\star}$ exits and is absolutely
  continuous with density
  \begin{align*}
    p^{\star}(x) = - \frac{b(x)}{\sigma^2_{B}(F)} \mathbb{E} \left[
      \frac{1}{b(F)} \frac{\gamma'(F)}{\gamma(F)}  \mathbb{I}[F \ge x]
    \right].  
  \end{align*}
  Moreover $F_{\infty}$ is the unique fixed point of this
  tranformation, in the sense that if ${F} \stackrel{\mathcal{D}}{=} F^{\star}$ then $F \stackrel{\mathcal{D}}{=} F_{\infty}$ (equality in distribution).
  
\end{lemma}

\begin{proof}
  The proof of unicity is easy; the other points follow from arguments nearly identical to those in \cite[Proposition 2.1]{ChGoSh11}.  \end{proof}

Now consider  the function 
\begin{align*}
  f(x) = \frac{1}{\gamma(x)} \int_{-\infty}^x \big(h(u) -
  \mathbb{E} h(F_{\infty})\big)  b(u) \gamma(u) \mathrm{d}u =:
  \frac{1}{\gamma(x)} \int_{-\infty}^x \bar h(u) b(u) \gamma(u)
  \mathrm{d}u
\end{align*}
which is solution to the differential equation
\begin{align*}
(\bar h(x) :=)  h(x) - \mathbb{E} h(F_{\infty}) = \frac{f'(x) + \gamma'(x)/\gamma(x)
  f(x)}{b(x)}
\end{align*}
for all $x \in \mathcal{S}_B$. 
Let $F$ be a random variable such that $P(F \in \mathcal{S}_B) = 1$ and  $\sigma^2_B (F)= 1$. We have
\begin{align}\label{eq:2app}
  \mathbb{E} h(F) - \mathbb{E} h(F_{\infty})    =  \mathbb{E} \left[
  \frac{f'(F)}{b(F)} + \frac{\gamma'(F)}{\gamma(F)}
  \frac{f(F)}{b(F)} \right] =  \mathbb{E} \left[
  \frac{f'(F)}{b(F)} -  
  \frac{f'(F^{\star})}{b(F^{\star})} \right]
\end{align}
and it remains to express the right hand side of \eqref{eq:2app} in terms
of manageable quantities, such as moments of $F, F^{\star}$ and
$F - F^{\star}$.  We cannot work directly with the function $x \mapsto f'(x)/b(x)$ because the latter is unbounded at $x = 0$. To bypass this difficulty,  we  introduce the notation
\begin{equation*}
  \mathcal{L}_{\infty}h(x) = \frac{1}{p_{\infty}(x)} \int_{-\infty}^x
  \bar h(u) p_{\infty}(u) \mathrm{d} u
\end{equation*}
(we stress that $\mathcal{L}_{\infty} \neq \mathcal{L}_{\gamma})$ and follow \cite{MPS19} by
introducing the function $g = g_{\eta, h}$ given by
\begin{align}\label{eq:6}
  g(x) = \frac{\mathcal{L}_{\infty}h(x) }{\mathcal{L}_{\infty}
  \eta(x)} =  \frac{ \int_{-\infty}^x (h(u) -
  \mathbb{E}h(F_{\infty})) b(u) \gamma(u) \mathrm{d}u}{  \int_{-\infty}^x (\eta(u) -
  \mathbb{E}\eta(F_{\infty})) b(u) \gamma(u) \mathrm{d}u}
\end{align}
at all $x \in \mathcal{S}_B$ where $h$ is fixed by the left hand side of \eqref{eq:2app} but $\eta$ is kept unspecified, to be tuned to our needs at a later stage. Obviously, the above relations are only defined  at $x$ 
such that $p_{\infty}(x) \neq  0$; we suppose this to be the case here and in the sequel. 
The  function $g$  from \eqref{eq:6} is then solution to  the Stein equation
\begin{align*}
  \left( \mathcal{L}_{\infty} \eta(x) \right) g'(x) +
  \bar{\eta}(x) g(x) = h(x) - \mathbb{E}h(F_{\infty})
\end{align*}
at all $x$ inside the support of $p_{\infty}$. 
It will   be useful to note that the functions $g, h$ and $\eta$ satisfy  the
relations
\begin{align}
  (\mathcal{L}_{\infty}\eta) g  &    = \mathcal{L}_{\infty}h \nonumber \\
  (\mathcal{L}_{\infty}\eta) g'  &   = \overline h - \overline \eta
    g=\overline h - \overline \eta
    \frac{\mathcal{L}_{\infty}h}{\mathcal{L}_{\infty}\eta} \nonumber
  \\
(\mathcal{L}_{\infty}\eta) g''    &  =  h' -  (\overline\eta +
    (\mathcal{L}_{\infty}\eta)') g' - \eta' g \label{eq:17bis}\\
  & = 
    \left(  \overline \eta
    \frac{\overline \eta + \left( \mathcal{L}_{\infty}\eta
    \right)'}{\mathcal{L}_{\infty}\eta} - \eta'\right)
    g - \left(  \overline h
    \frac{\overline \eta + \left( \mathcal{L}_{\infty}\eta
    \right)'}{\mathcal{L}_{\infty}\eta} - h'\right). \nonumber
\end{align}
Straightforward  manipulations of the definitions also
lead to
\begin{align*}
  \mathcal{L}_{\infty}h(x) = \frac{\mathcal{L}_{\gamma}(hb)(x)}{b(x)} -
  \mathbb{E}[h(Z)b(Z)]\frac{\mathcal{L}_{\gamma}b(x)}{b(x)}. 
\end{align*}
Finally we note that $g$ and $f$ are related through $f(x) = \mathcal{L}_{\infty} \eta(x) g(x) b(x)$ so that \begin{align*}
  \frac{f'(x)}{b(x)}= \mathcal{L}_{\infty} \eta(x)  g'(x)
  + \left(\overline\eta(x) - \frac{\gamma'(x)}{\gamma(x)}
  \mathcal{L}_{\infty} \eta(x)  \right)g(x) =: \mathcal{L}_{\infty} \eta(x)  g'(x)
                                                                                                                                                                                                                                        + \Psi_{\infty}\eta(x) g(x). 
                                                                                                                                                                                                                                      \end{align*} 
Identity \eqref{eq:2app} becomes
                                   \begin{align}
   \mathbb{E} h(F) - \mathbb{E} h(F_{\infty}) & =  \mathbb{E} \Big[
    \mathcal{L}_{\infty}
    \eta(F) 
    g'(F) - 
    \mathcal{L}_{\infty}
    \eta(F^{\star}) 
    g'(F^{\star})\Big]  \nonumber \\
  &  \quad + \mathbb{E}\Big[ \Psi_{\infty}\eta(F) g(F)
    -\Psi_{\infty}\eta(F^{\star})
    g(F^{\star}) \Big]\label{eq:20}
                                   \end{align}
                                   which is close to what is required. This is however not exactly what we need because, although we shall see that for reasonable choices of $\eta$, the
                                   function $g$ from \eqref{eq:6} and its derivative $g'$ are bounded,
                                   the second derivative $g''$ is often not.  In order to cater for this, we introduce some further degrees of liberty in the expressions and rewrite \eqref{eq:20} as
                                   \begin{align*}
  & \mathbb{E} h(F) - \mathbb{E} h(F_{\infty})  
    =  \mathbb{E} \Big[
    \Big( 
    r_1(F)\mathcal{L}_{\infty}
    \eta(F) - 
    r_1(F^{\star})\mathcal{L}_{\infty}
    \eta(F^{\star}) \Big)
    \frac{g'(F^{\star})}{r_1(F^{\star})}
    \Big]  \nonumber \\
  & \quad + \mathbb{E}\Big[
    r_1(F)\mathcal{L}_{\infty}
    \eta(F)
    \Big(\frac{g'(F)}{r_1(F)} - \frac{g'(F^{\star})}{r_1(F^{\star})}\Big) \Big] \nonumber\\
  & \quad + \mathbb{E} \Big[ \Big(r_2(F)\Psi_{\infty}\eta(F) -
    r_2(F^{\star})\Psi_{\infty}\eta(F^{\star})
    \Big) \frac{g(F^{\star})}{r_2(F^{\star})}  \Big]\nonumber\\
  & \quad + \mathbb{E}\Big[r_2(F)\Psi_{\infty}\eta(F) \Big(
    \frac{g(F)}{r_2(F)}-\frac{g(F^{\star})}{r_2(F^{\star})} \Big)\Big],
\end{align*}
with $r_1, r_2$ two functions left to be determined with the aim of
tempering the malevolent intentions of $g'$ and
$g''$. 
These considerations lead to the  main  result of the Section.
\begin{proposition}\label{prop:reexprg}
  Let the previous notations and assumptions prevail. Then
  \begin{align}
   &   \left|    \mathbb{E} h(F) - \mathbb{E} h(F_{\infty})
     \right|\nonumber \\
   &  \le \kappa_1 \mathbb{E} \Big[
    \Big| 
    r_1(F)\mathcal{L}_{\infty}
    \eta(F) - 
    r_1(F^{\star})\mathcal{L}_{\infty}
    \eta(F^{\star}) \Big|
    \Big]   + \kappa_2 \mathbb{E}\Big[
|    r_1(F)\mathcal{L}_{\infty}
    \eta(F)|
    |F - F^{\star}| \Big]  \nonumber  \\
\label{eq:24}    &  \quad + \kappa_3\mathbb{E} \Big[ \Big|r_2(F)\Psi_{\infty}\eta(F) -
    r_2(F^{\star})\Psi_{\infty}\eta(F^{\star})\Big|  \Big]  +
      \kappa_4\mathbb{E}\Big[|r_2(F)\Psi_{\infty}\eta(F)||F - F^{\star}|\Big] 
  \end{align}
  where $\kappa_j= \sup_x \left| \kappa_j(x) \right|$ for $j = 1, \ldots, 4$ with 
  \begin{align}
    \kappa_1(x) =      \frac{g'(x)}{r_1(x)},  
    \kappa_2(x) =  \left(\frac{g'(x)}{r_1(x)}   \right)', \, 
    \kappa_3(x) =  \frac{g(x)}{r_2(x)} \mbox{ and }
    \kappa_4(x) =     \left( \frac{g(x)}{r_2(x)} \right)'.\label{eq:23}
  \end{align}

\end{proposition}

\begin{remark}\label{rem:some-theory}
  The functions $r_1, r_2$ and $\eta$ can, for all intents and purposes, be chosen
freely.  A good choice of function $\eta$ seems to be
$\eta(x) = \eta_{\gamma}(x) = \gamma'(x)/\gamma(x)$, at least if
$\gamma'(x) / \gamma(x)$ is non-increasing on $\mathbb{R}$ and
$\mathbb{E} [b'(Z)] = 0$ when $Z \sim \gamma$.  Indeed in this case:
\begin{align*}
  &  \mathbb{E}[\eta_{\gamma}(F_{\infty})] =
                 -\mathbb{E}\Big[\frac{b'(F_{\infty})}{b(F_{\infty})}\Big]
  = - \mathbb{E} [b'(Z)] = 0 \mbox{ and } \overline\eta_{\gamma}(x) = \eta_{\gamma}(x)\\
  & \mathcal{L}_{\infty} \eta_{\gamma}(x)  
    =  1 -
    \frac{\mathcal{L}_{\gamma}b'(x)}{b(x)} \mbox{ and }
   \overline\eta_{\gamma}(x)  -\frac{\gamma'(x)}{\gamma(x)}
    \mathcal{L}_{\infty} \eta_{\gamma}(x) = \frac{\gamma'(x)}{\gamma(x)} \frac{\mathcal{L}_{\gamma}b'(x)}{b(x)}.
\end{align*}
Another natural choice (which turns out to be equivalent to the
previous one when $\gamma$ is the Gaussian density) is
$\eta(x) = - \mathrm{Id}(x) = -x$ for which
\begin{align*}
  &  - \mathbb{E}[\mathrm{Id}(F_{\infty})] =  - \mathbb{E}  \left[ Z
    b(Z) \right] = \mathbb{E}[b'(Z)]\\
   & - \mathcal{L}_{\infty} \mathrm{Id}(x) = \tau_{\infty}(x) \quad
     \mbox{ (the Stein kernel)}.
\end{align*}
Other choices are possible, depending on the properties of the density
$\gamma$; it may be worthwhile investigating this avenue, though we will not do it here. 
\end{remark}

\subsubsection{When the base distribution is standard Gaussian}
\label{sec:when-base-distr}

We now specialize the previous construction to the case that
$\gamma(x)$ is  the standard Gaussian density. As before, we suppose that $b$ is chosen in such a way that $\mathbb{E}[F_{\infty}] = 0$; note that if $Z\sim \gamma$ the standard normal then we also have $\mathbb{E}[F_{\infty}] = \mathbb{E}[Z b(Z)] = - \mathbb{E}[b'(Z)]$.  If $\gamma$ is the Gaussian density then many of the previous expressions simplify, because $\gamma'(x)/\gamma(x) = -\mathrm{Id}(x) :=-x $.  For instance $ \sigma^2_B(F) = \mathbb{E}[F {B(F)}/{b(F)}]$ and taking $\eta = - \mathrm{Id}$ we get $\mathcal{L}_{\gamma}\eta = 1$. Also $\mathcal{L}_{\infty}\eta = \tau_{\infty}$ is now the so-called Stein kernel of $p_{\infty}$; this function is well known to have very good properties for the analysis of $p_{\infty}$, see e.g.\ \cite{ernst2020first} for an overview. At this stage it suffices to remark that $\tau_{\infty}(x) \ge 0$ for all $x$. 
We also have the nice identity $ \Psi_{\infty}\eta(x) = x(\tau_{\infty}(x)-1)$ so that \eqref{eq:24} becomes
 \begin{align}
   &   \left|    \mathbb{E} h(F) - \mathbb{E} h(F_{\infty})
     \right|\nonumber \\
   &  \le \kappa_1 \mathbb{E} \Big[
     \Big| 
     r_1(F)\tau_{\infty}(F) - 
     r_1(F^{\star})\tau_{\infty}(F^{\star}) \Big|
     \Big]   + \kappa_2 \mathbb{E}\Big[
     |    r_1(F)|\tau_{\infty}(F)
     |F - F^{\star}| \Big]  \nonumber  \\
   &  \quad + \kappa_3\mathbb{E} \Big[ \Big|F r_2(F)
     (\tau_{\infty}(F) -1) -
     F^{\star}    r_2(F^{\star})(\tau_{\infty}(F^{\star}) -1)\Big|  \Big]
     \nonumber \\ \label{eq:8} 
  & \quad +  \kappa_4\mathbb{E}\Big[|Fr_2(F)(\tau_{\infty}(F) -1)||F - F^{\star}|\Big] 
 \end{align}
 with $\kappa_j, j=1, \ldots, 4$ still as in \eqref{eq:23}.  The
 following general result permits to bound these constants.
 \begin{lemma}\label{lem:boundedsol} Let all above notations and
   assumptions prevail (in particular $\|h'\| \le 1$).  Then
 \begin{align}
  & \label{eq:kappa1}\kappa_1(x) \le  \frac{2}{|r_1(x)|} \frac{R_{\infty}(x)}{(\tau_{\infty}(x))^2}\\
  & \nonumber \kappa_2(x) \le    
 \frac{2}{
   |r_1(x)|\tau_{\infty}(x) } \left( 1 +\left| \frac{2x}{\tau_\infty(x)} - x +
   \frac{b'(x)}{b(x)}-\frac{r_1'(x)}{r_1(x)}   \right| \frac{
     R_{\infty}(x) }{\tau_{\infty}(x)}   \right)\\
  &         \label{eq:kappa3}      \kappa_3(x)  \le    \frac{1}{|r_2(x)|} \\
  & \nonumber   \kappa_4(x)  \le  \frac{1}{|r_2(x)|}
     \left(    \frac{ 2R_{\infty}(x)}{(\tau_{\infty}(x))^2}  +
    \left|  \frac{r_2'(x)}{r_2(x)}  \right|\right)
\end{align}
where $ R_{\infty}(x) = { \int_{-\infty}^x{P}_{\infty}(u)
    \mathrm{d}u  \int_x^{\infty} \overline{P}_{\infty}(u)
                  \mathrm{d}u}/{p_{\infty}(x)}. $
\end{lemma}

\begin{proof}[Proof of Lemma \ref{lem:boundedsol}]
Let $g$ be defined in \eqref{eq:6} with $\eta = -\mathrm{Id}$. Suppose
that $\mathbb{E} b'(Z) = 0$ and let $h$ be absolutely continuous.
 We start with the
  fact that, from \eqref{eq:17bis}: 
\begin{align*}
 \tau_{\infty}(x) g''(x) =   h'(x) -(-x+ \tau_{\infty}'(x)) g'(x)  +
 g(x).
\end{align*}
Using
\begin{align*}
   \tau_{\infty}'(x) = \left( x - \frac{b'(x)}{b(x)} \right)
  \tau_{\infty}(x) - x 
\end{align*}
we get
\begin{align}\label{eq:gpripri}
    g''(x) = 
  \frac{g(x) + h'(x)}{\tau_{\infty}(x)}  + \left(
   \frac{2x}{\tau_{\infty}(x) } - x 
 + \frac{b'(x)}{b(x)}\right) g'(x).
\end{align}

With all this we are ready to bound the different coefficients, whose expressions we recall for ease of reference:
  \begin{align*}
    \kappa_1(x) =   \left|  \frac{g'(x)}{r_1(x)}  \right|,  
    \kappa_2(x)  = \left| \left(\frac{g'(x)}{r_1(x)}   \right)'\right|, \, 
    \kappa_3(x)  =  \left| \frac{g(x)}{r_2(x)}  \right| \mbox{ and }
    \kappa_4(x)  =  \left| 
    \left( \frac{g(x)}{r_2(x)} \right)' \right|.
  \end{align*}
  It follows immediately from \eqref{eq:13} that
  \begin{equation*}
    g(x) = \frac{\mathbb{E} \left[(h(X_1) - h(X_2)) \mathbb{I}[X_1 \le
        x \le X_2] \right]}{\mathbb{E} \left[(X_2 - X_1)
        \mathbb{I}[X_1 \le 
        x \le X_2] \right]}
  \end{equation*}
  at all $x \in \mathcal{S}(p_{\infty})$, where $X_1$ and $X_2$ are independent copies of $F_{\infty}$.
 Since, by assumption,
 $|h(x) - h(y)| \le \left| x-y\right|$, \eqref{eq:kappa3} follows.
 To pursue, we use \cite[Lemma 2.25]{ernst2020first} to deduce \begin{align*}
 g'(x) & = \frac{\bar h(x) \, \mathcal{L}_\infty\eta(x) - \bar \eta(x) \,
         \mathcal{L}_{\infty} h(x)}{ (\mathcal{L}_{\infty} \eta(x))^2}  \\
  & =  \frac{1}{p_{\infty}(x)  \tau_{\infty}(x)^2} \left( \mathbb{E} \left[
    h'(F_{\infty})F_{\infty}\frac{\bar P_{\infty}(F_{\infty})}{p_{\infty}(X)}
    \mathbb{I}[x \le F_{\infty}]\right]\mathbb{E} \left[
    \frac{ P_{\infty}(F_{\infty})}{p_{\infty}(F_{\infty})}
    \mathbb{I}[F_{\infty} \le x]\right] \right. \\
  & \qquad \left. -
    \mathbb{E} \left[
    h'(F_{\infty}) \frac{ P_{\infty}(F_{\infty})}{p_{\infty}(F_{\infty})}
    \mathbb{I}[F_{\infty} \le x]\right]
    \mathbb{E} \left[
     \frac{\bar P_{\infty}(F_{\infty})}{p_{\infty}(F_{\infty})}
    \mathbb{I}[x\le X]\right]
    \right)
\end{align*}
(where $\bar P$ is the survival function of cdf $P$).  The bound on the
derivative then follows (see also  \cite[Equation (2.38)]{ernst2020first}): 
\begin{align*}
 |g'(x)|  & \le  \|h'\| \frac{2}{p_{\infty} (x) 
\tau_{\infty}(x)^2}   \mathbb{E} \left[
  \frac{\overline{P}_{\infty}(F_{\infty})}{p_{\infty}(F_{\infty})}
  \mathbb{I}[x \le F_{\infty}] \right]  \mathbb{E} \left[
  \frac{{P}_{\infty}(F_{\infty})}{p_{\infty}(F_{\infty})}
  \mathbb{I}[F_{\infty}\le x]  
               \right] \\
  &\le  2 \frac{1}{\tau_{\infty}(x)^2}     \frac{ \int_{-\infty}^x{P}_{\infty}(u)
    \mathrm{d}u  \int_x^{\infty} \overline{P}_{\infty}(u)
    \mathrm{d}u}{p_{\infty}(x)}, 
\end{align*}
which brings \eqref{eq:kappa1}. Furthemore, simply by taking derivatives and using the previous bounds, we get 
 \begin{align*}
& \left| \left( \frac{g(x)}{r_2(x)} \right)' \right| \le  \frac{1}{|r_2(x)|}
     \left( \frac{ 2R_{\infty}(x)}{\tau_{\infty}(x)^2}   +
     \left| \frac{r_2'(x)}{r_2(x)} \right| \right) 
 \end{align*}
 as well as (using \eqref{eq:gpripri} to express $g''$ in terms of the lower order derivatives) 
 \begin{align*}
   \left( \frac{g'(x)}{r_1(x)} \right)' \le  \frac{2}{\tau_{\infty}(x)
   |r_1(x)| } \left( 1 +\left| \frac{2x}{\tau_\infty(x)} - x +
   \frac{b'(x)}{b(x)}-\frac{r_1'(x)}{r_1(x)}   \right| \frac{ R_{\infty}(x) }{\tau_{\infty}(x)}   \right).  
 \end{align*}
All  claims are therefore established. 
\end{proof}
  
We now apply these results to the choice
$B(x)\propto x |x|^{k}/(k+1)$ and $b(x) \propto |x|^k $ with
$k \in \mathbb{N}$. Then $\mathbb{E}F_{\infty} = 0$ and
\begin{align*}
  \sigma^2_B(F) = \mathbb{E} \left[ F \frac{B(F)}{b(F)} \right] =
  \frac{\mathbb{E}  F^2 }{k+1},                                                                                                                                 \end{align*}
so that our first assumptions  become  $P(F \neq 0) = 1$ and
$\mathbb{E}[F^2] = k+1$. The following     results then follow from
direct manipulations of the definitions. The first result is the same as Lemma 3.1 in the main text.

\begin{lemma} \label{lem:tau1inf} Let all above notations prevail, and set
  $\tau_{\infty}(\cdot; k)$ to be the Stein kernel of $p_{\infty}(x;, k) =
  b(x)\gamma(x)$.
If $x \neq 0$ then
  \begin{align*}
&    \tau_{\infty}(x; k) = 2^{k/2} e^{x^2/2} |x|^{-k} \Gamma(1+k/2,
                   x^2/2)
  \end{align*}
  where $\Gamma(\alpha, x) = \int_x^{\infty} t^{\alpha-1}e^{-t} \mathrm{d}t$ is the (upper) incomplete gamma function.
 \end{lemma}

 Incomplete gamma functions are quite well understood. For instance, using \cite{jameson2016incomplete},   we   
 readily obtain the next result.

  \begin{lemma}\label{lem:r1inf}
    The Stein kernel $\tau_{\infty}$ is strictly decreasing on $(0, \infty)$ and satisfies
  \begin{align*}
    \lim_{x \to 0} |x|^k \tau_{\infty}(x; k) = 2^{k/2} \Gamma(1+k/2) \mbox{ and }     \lim_{|x| \to \infty}  \tau_{\infty}(x; k) = 1
  \end{align*}
  as well as the inequalities 
  \begin{align}
    \label{eq:ineqontau1}
    &   \frac{1}{|x|^k \tau_{\infty}(x; k)} \le \frac{1}{2^{k/2}\Gamma(1+k/2)},  \quad    \left| \frac{2}{\tau_{\infty}(x; k) } - 1  \right| \le 1     \\
    &     \nonumber
 \mbox{ and }    \frac{1}{|x|^{k-1} \tau_{\infty}(x; k)} \le \frac{1}{2^{(k-1)/2}\Gamma(1+k/2)} 
  \end{align} 
  for all $x \in \mathbb{R}$ and all $k \in \mathbb{N}_0$. 
  Moreover, if we let $P_{\infty}(\cdot; k)$ and
  $\overline{P}_{\infty}(\cdot;
  k)$ be the cdf and survival function of
   $p_{\infty}(\cdot; k)$, and
   define $R_{\infty}(x;k) = { \int_{-\infty}^x{P}_{\infty}(u; k)
    \mathrm{d}u  \int_x^{\infty} \overline{P}_{\infty}(u; k)
    \mathrm{d}u}/{p_{\infty}(x; k)}$ as in Lemma \ref{lem:boundedsol}, then \begin{align}
                                                       \label{eq:boundrtau} \frac{R_{\infty}(x;k)}{\tau_{\infty}(x; k)} \le \frac{\Gamma (k/2+1)}{\sqrt2 \Gamma(k/2+1/2)}                      \end{align}             
                                                     for all $x \in \mathbb{R}$ and all $k \in \mathbb{N}_0$. 
               \end{lemma}
               \begin{remark}
                 The bounds \eqref{eq:ineqontau1} 
                 and \eqref{eq:boundrtau} are sharp because they are attained at $x \to 0$. 
               \end{remark}
               Using Lemmas \ref{lem:tau1inf} and \ref{lem:r1inf} we obtain the required bounds on the constants $\kappa_j$. 
               \begin{cor}\label{cor:boudnoaunfdaoun}
                 Set $r_1(x) = x^k$ and $r_2(x) =1$ (the sign of $x$) in Lemma~\ref{lem:boundedsol}.  Then
                          \begin{align}
  & \label{eq:kappa1b}\kappa_1(x) \le  \frac{2}{|x|^k} \frac{R_{\infty}(x;k)}{(\tau_{\infty}(x;k))^2}  \le \frac{2^{(1-k)/2}}{  \Gamma((k+1)/2)}\\
   & \label{eq:kappa2b} \kappa_2(x) \le    
  \frac{2}{|x|^k\tau_{\infty}(x;k)}  +\frac{2}{|x|^{k-1}\tau_{\infty}(x;k)  } \left| \frac{2}{\tau_{\infty; k}(x)} - 1  \right| \frac{R_{\infty}(x;k) }{\tau_{\infty}(x;k)}   \le 3 \frac{2^{-k/2}}{\Gamma((1+k)/2)}\\
  &         \label{eq:kappa3b}      \kappa_3(x)  \le  1  \\
  & \label{eq:kappa4b}  \kappa_4(x)  \le  
\frac{ 2R_{\infty}(x;k)}{(\tau_{\infty}(x;k))^2}   \le 1
\end{align}
 \end{cor}

 \begin{remark}
The bounds in \eqref{eq:kappa1b} and  \eqref{eq:kappa3b} are sharp; the other two are not.  
\end{remark}

\subsubsection{When $k=0, 1$ and a proof of Proposition
  \ref{prop:boundk1}
}
\label{sec:when-k=0-1}

Upper bounding \eqref{eq:kappa1b} by 1 and \eqref{eq:kappa2b} by 2
(neither of these choices, nor the bound 1 in \eqref{eq:kappa4b}, are
optimal in $k$ because the true value goes to 0 as $k$ goes to
$\infty$), inequality \eqref{eq:8} leads to the following result.
\begin{theorem}\label{prop:alternatotheore}
 If  $F_{\infty}\sim {\mathbb F}_\infty$ has density $p_{\infty}(x) \propto
|x|^{k} \varphi(x)$, for a given  $k \in \mathbb{N}$ and  $F\sim {\mathbb F}$ is some
random variable with mean 0 such that $P(F \neq 0) = 1$ and
$\mathbb{E}[F^2] = k+1$ then there exists a random variable
$F^{\star}$,  called  the generalized $k$-radial-bias transform of
$F$, which uniquely satisfies 
  \begin{equation*}
        \mathbb{E} \left[ \frac{f'(F^{\star})}{|F^{\star}|^k} \right] =  \mathbb{E} \left[ \frac{f(F)}{|F|^{k-1}}  \right]
  \end{equation*}
  for all $f$ such that both integrals exist, and
  \begin{align}
 \dw({\mathbb F},{ \mathbb F}_{\infty})&  \le \mathbb{E}  
     \Big| 
F^k\tau_{\infty}(F;k) - 
    (F^{\star})^k\tau_{\infty}(F^{\star};k) \Big|
         + 2 \mathbb{E} 
      \left[ |F|^k\tau_{\infty}(F;k)
     |F - F^{\star}|  \right]    \nonumber  \\
   &  \quad +  \mathbb{E} \Big|F  
     (\tau_{\infty}(F;k) -1) -
     F^{\star}   (\tau_{\infty}(F^{\star};k) -1)\Big|   
     \nonumber \\ 
  & \quad +   \mathbb{E} \left[ \big|F (\tau_{\infty}(F;k) -1)\big||F - F^{\star}| \right] \label{eq:tuinfaien}
   \end{align}
   where $\tau_\infty(x;k)$ is the Stein kernel given in \eqref{eq:tauinfk}.
%
  
\end{theorem}

The upper bounds in Propositions \ref{prop:boundk0} and
\ref{prop:boundk1} from the main text are direct corollaries of this
result.  We have already proved the upper bound from Proposition
\ref{prop:boundk0} in Example~\ref{prop:boundk0} by other means. We
therefore concentrate on Proposition \ref{prop:boundk1}.

The following  can be shown directly from the definitions:
  \begin{align*}
    (x \tau_{\infty}(x, 1))'\le 1, \quad | x (\tau_{\infty}(x, 1)-1)  | \le 1 \mbox{ and } |\left(  x (\tau_{\infty}(x, 1)-1)\right)' | \le 1.
  \end{align*}
  Plugging these into \eqref{eq:tuinfaien} gives
     \begin{align*}
     \dw({\mathbb F}, {\mathbb F}_{\infty}) & \le  3\mathbb{E}  \left[ | 
F  - 
    F^{\star} |  \right]   + 2 \mathbb{E} 
      \left[ (|F|(\tau_{\infty}(F;1)-1)
     |F - F^{\star}|  \right]   + 2 \mathbb{E} 
      \left[ |F| 
       |F - F^{\star}|  \right]  \\
     & = \mathbb{E}  \left[\left( 5 + 2 |F| \right) | 
 F - 
    F^{\star} |  \right]\\
    & \le (5+2 x_1)   (N-1)^{-1} \sum_{i=1}^{N-1}   |x_{i+1} - x_i| \\
    & = (5+2 x_1)  2{x_1}/(N-1)\\
   & =O(\log N/N),
   \end{align*}
  where we have used a  $k=1$  version of the coupling argument given
  in \cite[Proof of Corollary 3.7]{MPS19} along with  a version of
  \cite[Lemma 4.8]{MPS19} showing that  $x_1=O(\sqrt{\log N})$, as
  required.

\subsubsection{When $k\ge2$}

We now focus on the case $k \ge 2$.  We once again call upon
\cite{jameson2016incomplete} to obtain the next lemma.

 \begin{lemma}[Stein kernel] \label{sec:appr-radi-distr-1} Let all previous notations prevail. Given $j, k$ two integers define
   \begin{equation*}
a_j(k) =  {2^j}    \frac{\Gamma(1+k/2)}{\Gamma(1+k/2-j)} 
\end{equation*}
with the convention that $a_j(k) = 0$ for all $j \ge k$. 
 Then,  for all $x \in \mathbb{R}$ and all $k \in \mathbb{N}$, we have 
  \begin{equation*}
    \tau_{\infty}(x; k) =   \sum_{j=0}^{\lfloor \frac{k}{2}\rfloor} \frac{a_j(k) }{x^{2j}} + \frac{a_{\lceil k/2\rceil }(k) }{\sqrt2} \epsilon_{k}(x)  \end{equation*}
  where
  \begin{equation*}
    \epsilon_k(x) =
    \begin{cases}
      0 & \mbox{ if } k = 2 \ell  \\
 e^{x^2/2} |x|^{-(2 \ell -1)}
  \Gamma(1/2, x^2/2) & \mbox{ if } k = 2 \ell -1 
    \end{cases}
  \end{equation*}
  Moreover the remainder satisfies 
  $$0 \le \epsilon_{k}(x) \le 2^{(k+1)/2} x^{- (k+1)}  \frac{\Gamma(1+k/2)}{\Gamma(1/2)}$$
  for all $k \in \mathbb{N}$ and all $x$. 
\end{lemma}

\begin{proof}
  The claim follows from the following representation of the incomplete gamma function (available e.g.\ from \cite[Theorem 3 and Proposition 13]{jameson2016incomplete}): for all $a > 0$ and all $x>0$,
 \begin{align*}
   &     \Gamma(a, x) = e^{-x} x^{a-1}\sum_{j=0}^{\lfloor a \rfloor -1} P_j(a) x^{-j} + r(a, x)
 \end{align*}
 where $P_j(a) = {\Gamma(a)}/{\Gamma(a-j)}$ (and $P_j(a) = 0$ for all $j \ge a$), and, setting $[a] = a-\lfloor a \rfloor $, $ r(a, x) = P_{\lfloor a\rfloor}(a)\Gamma([a] , x)$ which satisfies
 \begin{align*}
   0 \le r(a, x) \le e^{-x} P_{\lfloor a \rfloor}(a) x^{[a] -1}.
 \end{align*}
The claim follows.

\end{proof}
In \cite{MPS19} we considered $b(x)=x^2$. The argument from that paper
is now extended to arbitrary integer $k$ in the next result.

 \begin{theorem}\label{thm:wassff}
  
Instate all previous notations and let    $$a_r(k) = \frac{\Gamma(k/2+1)}{\Gamma(k/2-r+1)}  2^r.$$  It holds that
 if $k = 2 \ell$ is an even integer then
 \begin{align*}
d_W(F, F_{\infty})  & \le  \sum_{j=0}^{\ell} a_{\ell - j}(2\ell) \Bigg(    \mathbb{E}   \left|  F^{2j} - 
                                               (F^{\star})^{2j}\right| + 2 \mathbb{E}   \
\left[ |F|^{2j}
                                               \left|F - F^{\star}
                                               \right|  \right] \Bigg)  \nonumber \\
  & \quad \qquad +  \sum_{j=1}^{\ell} a_{j}(2\ell) \Bigg(     \mathbb{E} 
                                               \left|  \frac{1}{F^{2j-1}}
                                -  \frac{1}{(F^{\star})^{2j-1}}
                                               \right|
                                                                 + \mathbb{E}\left[
                                               \frac{1}{|F|^{2j-1}}
                                               \Big|F-F^{\star}\Big|
    \right] \Bigg)                \end{align*}         
and,   if $k = 2 \ell -1$ is an odd integer, then 
                      \begin{align*}
d_W(F, F_{\infty})  & \le   \sum_{j=1}^{\ell} a_{\ell - j}(2\ell-1) \Bigg(    \mathbb{E}   \left|  F^{2j-1} - 
                                              (F^{\star})^{2j-1}\right| + 2 \mathbb{E}   \
\left[ |F|^{2j-1}
                                               \left|F - F^{\star}
                                               \right|  \right] \Bigg)  \nonumber \\
  & \quad \qquad +  \sum_{j=1}^{\ell-1} a_{j}(2\ell-1) \Bigg(     \mathbb{E} 
                                               \left|  \frac{1}{F^{2j-1}}
                                -  \frac{1}{(F^{\star})^{2j-1}}
                                               \right|
                                                                 + \mathbb{E}\left[
                                               \frac{1}{|F|^{2j-1}}
                                               \Big|F-F^{\star}\Big|
    \right] \Bigg)\nonumber \\
  & \quad \qquad + 3 a_{\ell}(2\ell-1)  \mathbb{E}\left[  \left(  2+ \frac{2}{|F|^{2(\ell-1)}}+ \frac{1}{|F^{\star}|^{2(\ell-1)}}\right)\left| F - F^{\star} \right|\right]  \nonumber \\
  & \quad \qquad + 3 a_{\ell}(2\ell-1) \mathbb{E} \left[ \left(  |F| + |F|^{2 \ell -1}  \right) \left|\frac{1}{F^{2\ell-1}} - \frac{1}{(F^{\star})^{2\ell-1}} \right|    \right]  \end{align*}
\end{theorem}

\begin{proof}
  The claim for even integers $k$ is immediate. If $k = 2 \ell-1$ is
  an odd integer then
      \begin{align*}
d_W(F, F_{\infty})  & \le  \sum_{j=1}^{\ell} a_{\ell - j}(2\ell-1) \Bigg(    \mathbb{E}   \left|  F^{2j-1} - 
                                              (F^{\star})^{2j-1}\right| + 2 \mathbb{E}   \
\left[ |F|^{2j-1}
                                               \left|F - F^{\star}
                                               \right|  \right] \Bigg)  \nonumber \\
  & \quad \qquad +  \sum_{j=1}^{\ell-1} a_{j}(2\ell-1) \Bigg(     \mathbb{E} 
                                               \left|  \frac{1}{F^{2j-1}}
                                -  \frac{1}{(F^{\star})^{2j-1}}
                                               \right|
                                                                 + \mathbb{E}\left[
                                               \frac{1}{|F|^{2j-1}}
                                               \Big|F-F^{\star}\Big|
    \right] \Bigg) \\
        & \quad \qquad + \frac{a_{\ell}(2\ell-1)}{\sqrt2}  \Psi_{\ell}(F, F^{\star})           \end{align*}
  where 
  \begin{align*}
    \Psi_{\ell}(F, F^{\star}) & =   \mathbb{E}\left|  F^{2\ell-1} \epsilon_{2\ell-1}(F)-(F^{\star})^{2\ell-1}   \epsilon_{2\ell-1}(F^{\star}) \right|   + 2  \mathbb{E}\left[  |F|^{2\ell-1} \epsilon_{2\ell-1}(F) |F - F^{\star}|\right] \\
                              & \quad  + \mathbb{E}\left|  F  \epsilon_{2\ell-1}(F)-(F^{\star})  \epsilon_{2\ell-1}(F^{\star}) \right|  + \mathbb{E}\left[  |F|  \epsilon_{2\ell-1}(F) |F - F^{\star}|\right]                              
  \end{align*}

  Here we aim to bound 
  \begin{align*}
    \Psi_{\ell}(F, F^{\star}) & =   \mathbb{E}\left|  F^{2\ell-1} \epsilon_{2\ell-1}(F)-(F^{\star})^{2\ell-1}   \epsilon_{2\ell-1}(F^{\star}) \right|   + 2  \mathbb{E}\left[  |F|^{2\ell-1} \epsilon_{2\ell-1}(F) |F - F^{\star}|\right] \\
                              & \quad  + \mathbb{E}\left|  F  \epsilon_{2\ell-1}(F)-(F^{\star})  \epsilon_{2\ell-1}(F^{\star}) \right|  + \mathbb{E}\left[  |F|  \epsilon_{2\ell-1}(F) |F - F^{\star}|\right]     \\
    & =: I + II + III + IV
  \end{align*}
where $  \epsilon_{2\ell -1}(x) =
 e^{x^2/2} |x|^{-(2 \ell -1)}
 \Gamma(1/2, x^2/2)$.  We first note that, for $x>0$, the function $x\mapsto \nu(x) := e^{x^2/2} \Gamma(1/2, x^{2/2})$ is strictly decreasing as $|x| \to \infty$, with maximum value $\sqrt \pi$ at $x=0$.  Hence
  \begin{align*}
 II + IV  & \le   2  \sqrt{\pi} \mathbb{E}\left[   |F - F^{\star}|\right]  + \sqrt \pi  \mathbb{E}\left[  \frac{1}{|F|^{2(\ell -1)}}  |F - F^{\star}|\right].                              
  \end{align*}
  Also, $|\nu'(x)| \le \sqrt 2$ so that $|\nu(x) - \nu(y)| \le \sqrt{2} |y - x|$, hence
  \begin{align*}
    \left|  \epsilon_{2\ell-1}(x) -   \epsilon_{2\ell-1}(y)  \right| &  \le |x|^{-(2\ell -1)} |\nu(x) - \nu(y)| + |x^{-(2\ell -1)} - y^{-(2\ell -1)}| \nu(y) \nonumber \\
    & \le \sqrt{2}|x|^{-(2\ell -1)} |x-y| + \sqrt\pi |x^{-(2\ell -1)} - y^{-(2\ell -1)}|. 
  \end{align*}
  This gives 
  \begin{align*}
    I & \le  \mathbb{E}\left[ \left|  F \right|^{2\ell-1} \left|  \epsilon_{2\ell-1}(F) -   \epsilon_{2\ell-1}(F^{\star})  \right|  \right] +
         \mathbb{E} \left[ \left|F^{2\ell-1} - (F^{\star})^{2\ell-1} \right| \epsilon_{2\ell-1}(F^{\star})    \right] \\
    & \le  \mathbb{E}\left[ \left|  F \right|^{2\ell-1} \left|  \epsilon_{2\ell-1}(F) -   \epsilon_{2\ell-1}(F^{\star})  \right|  \right]  + \sqrt \pi \mathbb{E} \left[\frac{1}{ |F^{\star}|^{2 \ell -1}} \left|F^{2\ell-1} - (F^{\star})^{2\ell-1} \right|    \right] \\
   & \le \sqrt 2 \mathbb{E}\left[ \left|  F - F^{\star}  \right|  \right]  + \sqrt \pi \mathbb{E} \left[  |F|^{2 \ell -1} \left|F^{-(2\ell-1)} - (F^{\star})^{-(2\ell-1)} \right|    \right]\\
   & \qquad +  \sqrt \pi \mathbb{E} \left[\frac{1}{ |F^{\star}|^{2 \ell -1}} \left|F^{2\ell-1} - (F^{\star})^{2\ell-1} \right|    \right]. 
 \end{align*}
 Similarly,
 \begin{align*}
   III & \le \mathbb{E}\left[ \left|  F \right|  \left|  \epsilon_{2\ell-1}(F) -   \epsilon_{2\ell-1}(F^{\star})  \right|  \right] +
         \mathbb{E} \left[ \left|F  - F^{\star} \right| \epsilon_{2\ell-1}(F^{\star})    \right] \\
       & \le  \mathbb{E}\left[ \left|  F \right|  \left|  \epsilon_{2\ell-1}(F) -   \epsilon_{2\ell-1}(F^{\star})  \right|  \right]    + \sqrt \pi \mathbb{E} \left[  \frac{1}{ |F^{\star}|^{2 \ell -1}}  \left|F  - F^{\star} \right|    \right] \\
       & \le \sqrt 2 \mathbb{E}\left[ \left|  F \right|^{-2(\ell  - 1)} \left|  F - F^{\star}  \right|  \right]  + \sqrt \pi \mathbb{E} \left[  |F| \left|F^{-(2\ell-1)} - (F^{\star})^{-(2\ell-1)} \right|    \right]\\
   & \quad  + \sqrt \pi \mathbb{E} \left[  \frac{1}{ |F^{\star}|^{2 \ell -1}}  \left|F  - F^{\star} \right|    \right].
 \end{align*}
 Combining these bounds leads to
 \begin{align*}
   \Psi_{\ell}(F, F^{\star}) & \le    2  \sqrt{\pi} \mathbb{E}\left[   |F - F^{\star}|\right]  + \sqrt \pi  \mathbb{E}\left[  \frac{1}{|F|^{2(\ell -1)}}  |F - F^{\star}|\right]\\
   & \qquad + \sqrt 2 \mathbb{E}\left[ \left|  F - F^{\star}  \right|  \right]  + \sqrt \pi \mathbb{E} \left[  |F|^{2 \ell -1} \left|F^{-(2\ell-1)} - (F^{\star})^{-(2\ell-1)} \right|    \right]\\
   & \qquad +  \sqrt \pi \mathbb{E} \left[\frac{1}{ |F^{\star}|^{2 \ell -1}} \left|F^{2\ell-1} - (F^{\star})^{2\ell-1} \right|    \right] \\
   & \qquad +\sqrt 2 \mathbb{E}\left[ \left|  F \right|^{-2(\ell  - 1)} \left|  F - F^{\star}  \right|  \right]  + \sqrt \pi \mathbb{E} \left[  |F| \left|F^{-(2\ell-1)} - (F^{\star})^{-(2\ell-1)} \right|    \right]\\
   & \qquad  + \sqrt \pi \mathbb{E} \left[  \frac{1}{ |F^{\star}|^{2 \ell -1}}  \left|F  - F^{\star} \right|    \right]
 \end{align*}
 which, after bounding all constants by $3 \sqrt{2}$ for simplicity, gives
 \begin{align*}
    \Psi_{\ell}(F, F^{\star}) & \le 3 \sqrt{2}  \mathbb{E}\left[  \left(  2+ \frac{2}{|F|^{2(\ell-1)}}+ \frac{1}{|F^{\star}|^{2(\ell-1)}}\right)\left| F - F^{\star} \right|\right]  \\
       & + 3 \sqrt{2} \mathbb{E} \left[ \left(  |F| + |F|^{2 \ell -1}  \right) \left|F^{-(2\ell-1)} - (F^{\star})^{-(2\ell-1)} \right|    \right], 
 \end{align*}
 which is the claim. 
\end{proof}
 
\begin{cor} 
  Instate all notations from Theorem \ref{thm:wassff}. The following bounds hold:
  \begin{itemize}

          \item $k=2$ (Maxwell case):  $p_{\infty}(x) = x^2 \varphi(x)$, $\mathbb{E}[F^2] = 3$, $\mathbb{E}[f'(F^{\star})/(F^{\star})^2] = \mathbb{E}[f(F)/|F|]$ for all $f$, and 
    \begin{align*}
       d_W(F, F_{\infty}) & \le \mathbb{E} \left[  \left(3 + 2 |F| + \frac{2}{|F|} + \frac{2}{|F| |F^{\star}|}\right) |F - F^{\star}|\right].
    \end{align*}    

            \item $k=3$:  $p_{\infty}(x) \propto |x|^3 \varphi(x)$, $\mathbb{E}[F^2] = 4$, $\mathbb{E}[f'(F^{\star})/|F^{\star}|^3] = \mathbb{E}[f(F)/F^2]$ for all $f$, and 
    \begin{align*}
      d_W(F, F_{\infty}) & \le \mathbb{E} \left[  \left(21
                           + 6 |F| + 2F^2 +  \frac{3}{|F|} + \frac{18}{F^{2}} + \frac{9}{(F^{\star})^2}+ \frac{3}{|F| |F^{\star}|}\right) |F - F^{\star}|\right]\nonumber \\
                               & \quad + \mathbb{E}|F^2 - (F^{\star})^2| + 9\mathbb{E} \left[(|F|+|F|^3) \left|  \frac{1}{F^3} - \frac{1}{(F^{\star})^3} \right| \right]. 
    \end{align*}    

  \end{itemize}

\end{cor}

\begin{remark}
We note that the Maxwell bound ($k=2$) is the same as in
\cite[Equation (24)]{MPS19} but with improved constants.  
\end{remark}

\subsection{Upper bounds on the rate of convergence via coupling}
\label{sec:conv-high-order}

 In this section we apply Theorem \ref{thm:wassff} to give explicit upper bounds on the accuracy of  $F=F_N\sim {\mathbb F}_N$ (defined after Lemma 3.3 of the main text) in approximating the radial distributions with density $p_{\infty}(x) \propto |x|^{k} \varphi(x)$ for $k=2,\ldots,14$.

 Two crucial steps in controlling the various ``error" terms in Theorem \ref{thm:wassff} when $k\ge 1$ (in which case there is a singularity in $p_\infty$ at the origin) are to obtain good approximations to $x_1$ and $x_m$ (where $m=N/2$), which (when necessary) we write as $  x_{1,k}$ and $  x_{m,k}$, to reflect their dependence on $k$.  It is easily shown by extending \cite[Lemma 4.8]{MPS19} 
 that
\begin{align*}
 x_{1,k} = O(\sqrt 
 {\log N})
\end{align*}
for all $k\ge 0$.  Concerning $  x_{m,k}$,  we have  $x_{m,k} \ge 1/\sqrt N$ for all  $k\ge 1$, cf.\ \cite[proof of Corollary 3.7]{MPS19}.  However, for $k\ge 3$ a lower bound of order $1/\sqrt N$ is too small to control all error terms in Theorem \ref{thm:wassff}, and it seems very challenging to improve this lower bound analytically.  

A way around this difficulty is to examine the numerical behavior of $x_{m,k}$.  We find  that the following scaling relation provides a remarkably accurate approximation:
 \begin{equation}
\label{ratexm}
x_{m,k} \asymp N^{-1/r_k}, \ \ \ {\rm where} \ r_k=3/2+4k/5 
\end{equation}
for $k=2,3,\ldots,14$, see Figure \ref{r_kplot}.  We expect this scaling relation hold for all $k\ge 2$.

\begin{figure}[!ht]
\begin{center}
 \includegraphics[scale=.4]{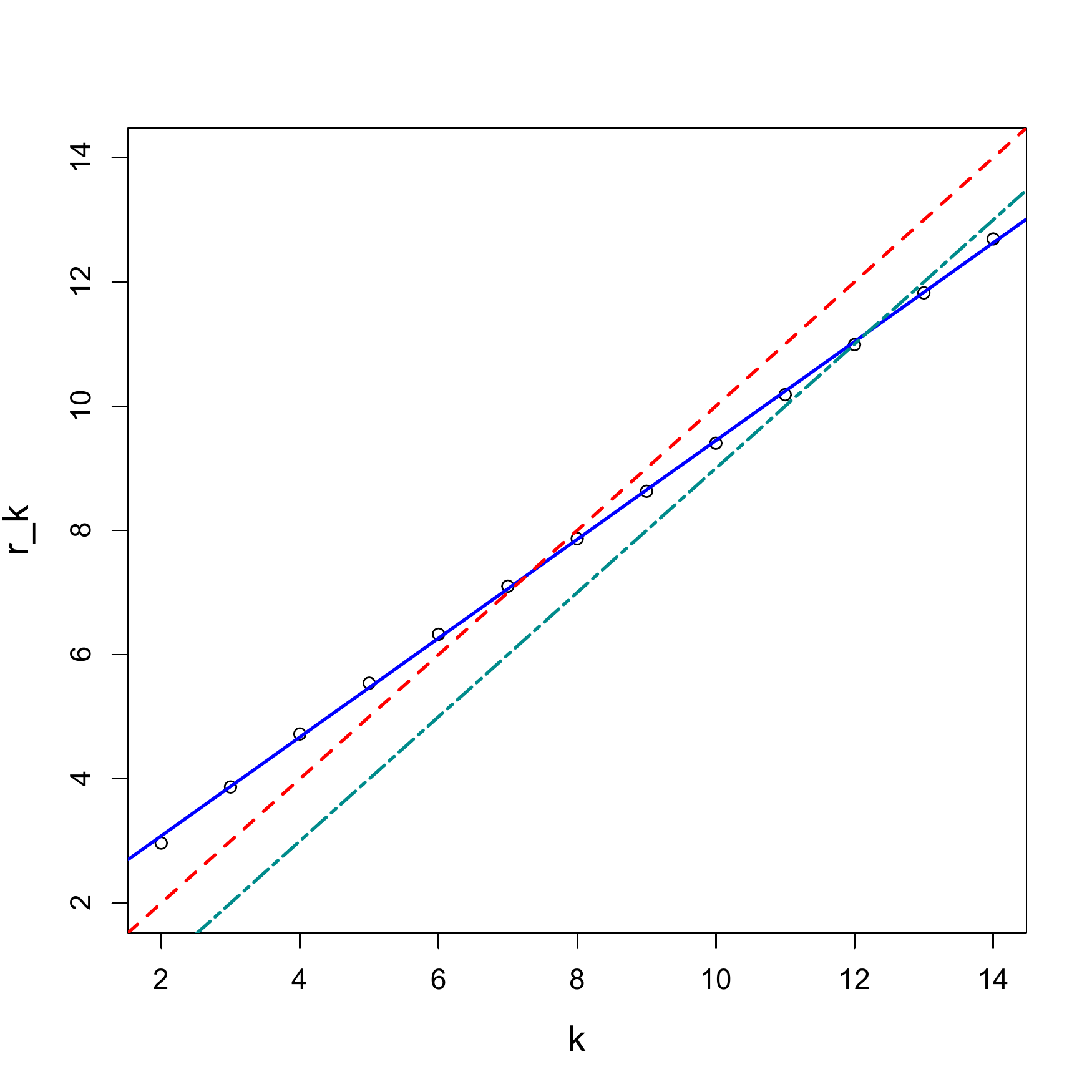} 
 \caption{The graph of $r_k=3/2+4k/5$  (blue solid line) obtained by fitting the scaling relation \eqref{ratexm} for  $k=2,\ldots,14$.  The fitted line approximates data points  (circles)  derived from  estimates of the slope when regressing $ \log x_{m,k}$ against $\log N$, $N=14, 24, \ldots, 114$, separately for each value of $k$.  It is remarkable that there is virtually no scatter around any of the linear fits.  For  $k$ odd, the fitted line  needs to fall above the diagonal (red dashed) line to ensure adequate control of \eqref{termhard} when $l=k$.  For $k$ even, 
 the fitted line needs to fall above the dot-dashed (cyan) line to obtain control when $l=k-1$.}
   \label{r_kplot}
 \end{center}
   \end{figure}

The next step is to  use the above scaling relation to obtain upper bounds on the  error terms in Theorem \ref{thm:wassff}  of the form  
\begin{equation}
\label{termhard}
 \mathbb{E} 
                                               \left|  \frac{1}{F^{l}}
                                -  \frac{1}{(F^{\star})^{l}}
                                               \right|, \ \ 1\le l\le k
\end{equation}
with  
$F\sim \mathbb{F}_N$, the  empirical distribution  of $x_1>\ldots >x_N$, and $F^{\star}\sim \mathbb{F}^{\star}_N$ is the corresponding $k$-radial-bias distribution, as given in the following lemma, which is a consequence of  \cite[Proposition 3.5]{MPS19}.

   \begin{lemma} \label{prop:gener-zero-bias}
The $k$-radial-bias  distribution  ${\mathbb{F}}^{\star}_N$ of $\mathbb{F}_N$ is defined, and has density
$$p^{\star}(x)\propto  |x|^k\left[\sum_{i=1}^n \frac{x_i}{|x_i|^k}\right] $$     
for  $x_{n+1}<  x \le x_{n}$ ($n=1,\ldots, N-1$), and $p^{\star}(x) = 0$ if $x> x_1$ or  $x\le x_N$.  \end{lemma}

 From Lemma \ref{prop:gener-zero-bias} and the recursion satisfied by $x_1>\ldots >x_N$, it follows  that $p^{\star}(x)$ puts  mass $1/(N-1)$ on each interval between successive $x_n$, so there exists a coupling of $F \sim \mathbb{P}_N$ with $F^{\star}\sim p^{\star}(x)$ 
such that  $$|F- F^{\star}|\le  |x_{n}-x_{n+1}|$$ when $F\in
[x_{n+1},x_n]$. For a detailed proof of such a coupling, see the
construction given in \cite{mckeague2016}.  
Now decompose \eqref{termhard} as 
$$ \IE\left| \frac{1}{F^l}-\frac{1}{(F^{\star})^l}\right|= \IE \left|\frac{1}{F^l}-\frac{1}{(F^{\star})^l}\right| 1_{F^{\star}\in (x_{m+1}, x_m]} + 2\sum_{n=1}^{m-1}  \IE \left|\frac{1}{F^l}-\frac{1}{(F^{\star})^l}\right| 1_{F^{\star}\in (x_{n+1}, x_n]}. $$
  From Proposition \ref{prop:gener-zero-bias} note that $p^{\star}(x) \propto |x|^k$  for $x\in (x_{m+1}, x_m]$. Using the fact  that $p^{\star}(x)$ puts  mass $1/(N-1)$ on this  interval,  the first term above can be written
$$\frac{2(k+1)}{x_m^{k+1}(N-1)}\int_0^{x_m} \left({1\over x^{l}}-{1\over x_m^{l}}\right)x^k\, dx\asymp  \frac{1}{x_{m}^{l}N}\asymp N^{l/r_k-1}\to 0,$$
provided $l< 3/2 +4k/5$ by \eqref{ratexm}.
The second term is bounded above by the telescoping sum
$${2\over N-1}\sum_{n=1}^{m-1}\left({1\over x_{n+1}^l}-{1\over x_n^l}\right)={2\over N-1} \left({1\over x_{m}^l}-{1\over x_1^l}\right)=O\left(N^{l/r_k-1}\right),$$
so we conclude
$$ \IE\left| \frac{1}{F^l}-\frac{1}{(F^{\star})^l}\right|=O\left(N^{l/r_k-1}\right).$$
This bound gives the desired control of \eqref{termhard}  for any
$l\le k\le 7$, see Figure \ref{r_kplot}.  For even $k$, we only need
to consider $l\le k-1$, so we have control for $k=8, 10$ and $12$, as
well.

\pagebreak

\section{Codes}
\subsection{Mathematica code}

\includepdf[page={-}]{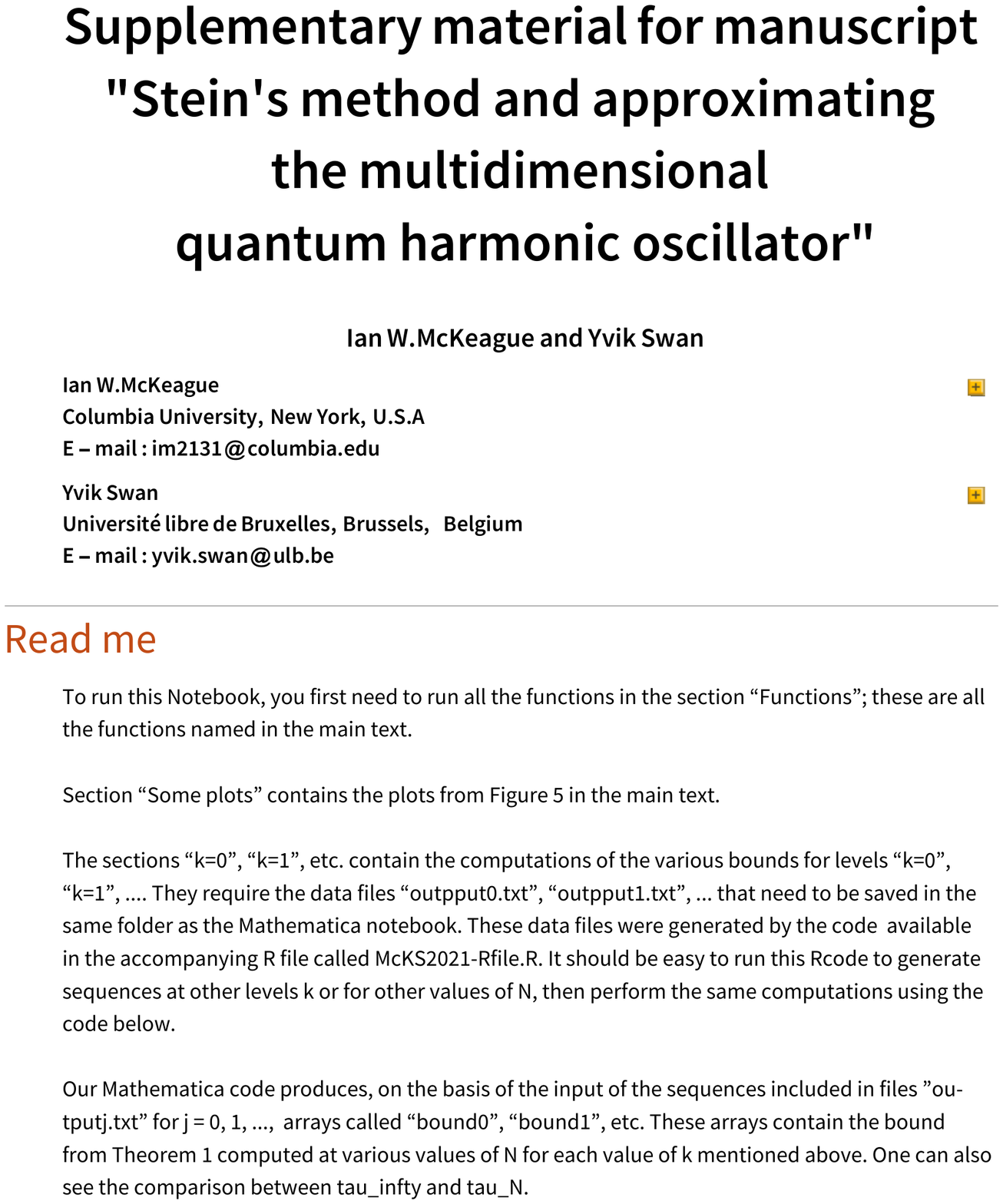}

\subsection{R Code}

\includepdf[page={-}]{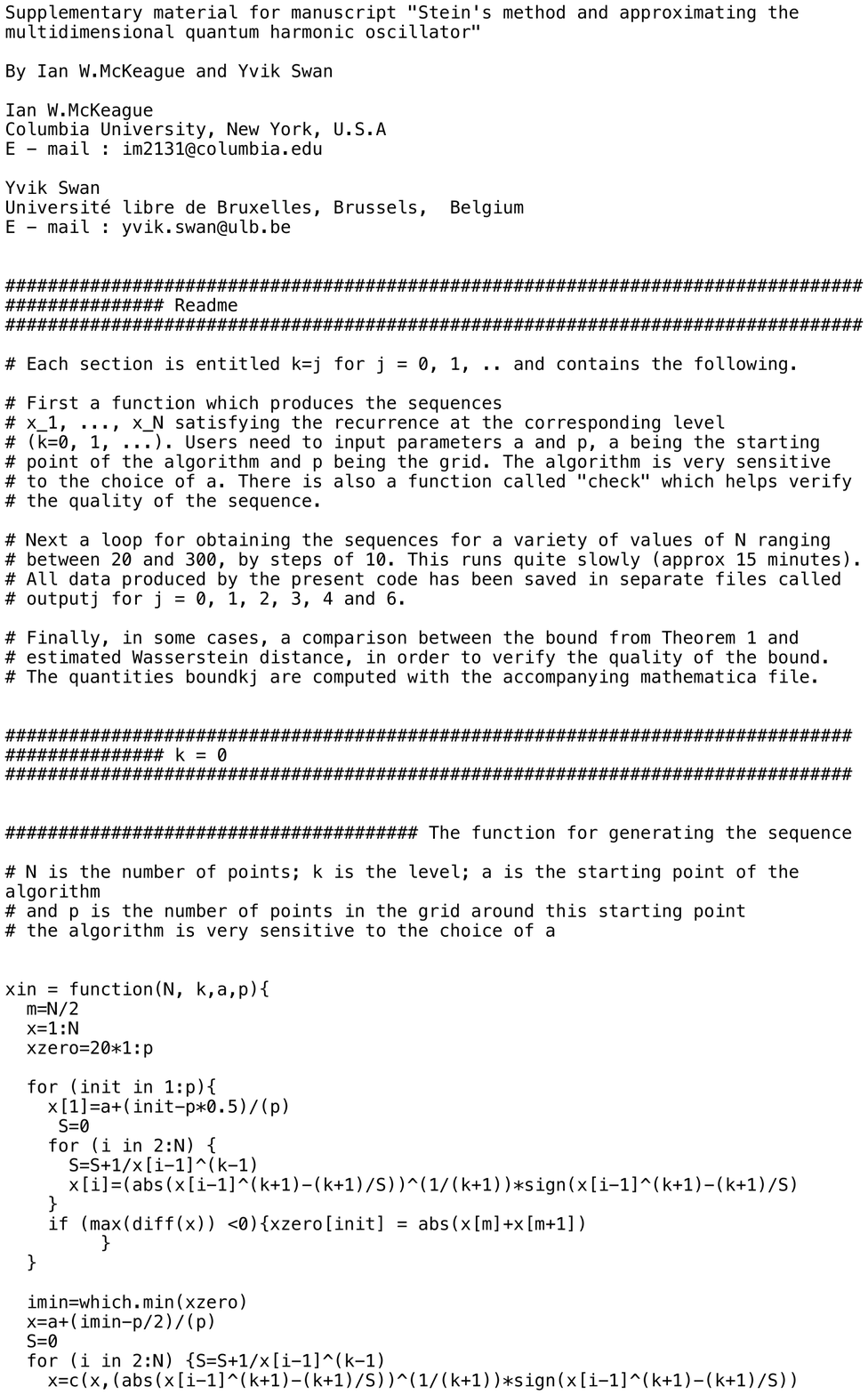}

\end{document}